\documentclass[11pt, a4paper]{article}
\usepackage[left=3cm, right=3cm, top=3.5cm, bottom=3.5cm]{geometry}
\usepackage[justification=justified, font=small, format=hang, margin={1cm,1cm}]{caption}
\usepackage[affil-it]{authblk}
\usepackage{amsmath}
\usepackage{amsthm}
\usepackage{amssymb}
\usepackage{booktabs} 
\usepackage{graphicx}
\usepackage{bbm} 
\usepackage{here}
\usepackage[round]{natbib}
\usepackage{hyperref}
\usepackage{titlesec}
\usepackage[font=footnotesize]{subfig}
\usepackage{algorithm}
\usepackage{algorithmic}
\usepackage{enumitem}
\usepackage[english, status=final]{fixme} 
\usepackage{caption}

\titleformat*{\section}{\Large\bfseries}
\titleformat*{\subsection}{\bfseries}
\titleformat*{\subsubsection}{\bfseries}
\titleformat*{\paragraph}{\bfseries}

\newcommand{\cR}{{\cal R}}

\newcommand{\mN}{\mathbb{N}}
\newcommand{\mR}{\mathbb{R}}

\newcommand{\bY}{\mathbf{Y}}

\newcommand{\Var}{{\rm Var}}

\theoremstyle{plain}
\newtheorem{theorem}{Theorem}[section]

\theoremstyle{definition}
\newtheorem{definition}[theorem]{Definition}

\theoremstyle{remark}
\newtheorem{remark}[theorem]{Remark}

{\bf}{\it}
{\bf}{\it}
{\bf}{\it}
{\bf}{\it}
{\bf}{\it}
{\bf}{\it}
{\bf}{\it}
{\bf}{\it}
{\bf}{\it}
{\bf}{\it}

\DeclareMathOperator*{\argmax}{arg\,max}

\newcommand{\one}{{\mathbbm{1}}}

\makeatletter
\def\@maketitle{%
  \newpage
  \null
  \vskip 2em%
  \begin{center}%
  \let \footnote \thanks
    {\Large\bfseries \@title \par}%
    \vskip 1.5em%
    {\normalsize
      \lineskip .5em%
      \begin{tabular}[t]{c}%
        \@author
      \end{tabular}\par}%
    \vskip 1em%
    {\normalsize \@date}%
  \end{center}%
  \par
  \vskip 1.5em}
\makeatother

\numberwithin{equation}{section}

\begin{document}

\title{Algorithm for overlapping estimation of common change-sets in spatial data of fixed size}

\author{Leonid Torgovitski\thanks{E-mail: \texttt{ltorgovi@math.uni-koeln.de}} \thanks{Research partially supported by the Friedrich Ebert Foundation.}}
\affil{Mathematical Institute, University of Cologne\\ Weyertal 86-90, 50931, Cologne, Germany}
\date{}

\maketitle

\abstract{We propose a flexible class of estimates for ``common change in the mean'' sets in spatio-temporal data. We rely on a scan type approach by subdividing the spatial observations into suitable overlapping regions to which classical CUSUM  (cumulative sums) estimates may then be applied separately. The aggregated ``local'' estimates are used to construct consistent ``global''  estimates of the change set(s) by taking the overlapping structure into account. The domain and the change regions may have irregular shapes and the suggested procedure is especially suited for estimation of multiple change regions. The performance is demonstrated in a simulation study.}

\subsubsection*{Keywords}
Change-set estimation, Spatio-temporal data, Fixed time, Common change-sets, Multiple changes
 

\section{Introduction}\label{sec:introduction}

The estimation of common change points (e.g. points which correspond to changes of the mean) in the framework of multivariate time series ~$X_1,X_2,\ldots,X_N\in\mR^d$, with fixed sample size ~$N$ ~and an increasing dimension ~$d\rightarrow\infty$, received some attention in recent literature (cf., e.g., \citet{Bai2010}, \citet{vert2010}, \citet{rao2012}, \citet{kim2014} and \cite{torg2014v2}). Within such a setting, \citet{Bai2010} studied a {\it single common change point} model and considered a classical least squares estimate, whereas \citet{vert2010,vert2011a} considered a {\it multiple} common change point model and adapted the {\it total variation denoising} approach to it. As pointed out in \cite{torg2014v2}, both methods may be seen as special cases of some general class of weighted CUSUM (cumulative sums) estimates. In the latter article, the consistency is studied for a whole class of such estimates which will also play a major role in the present work.

In this paper we turn to time series of spatial data ~$X_1,X_2,\ldots,X_d\in\mR^{m\times n}$, where the parameters ~$m,n$ ~are fixed and ~$d\rightarrow\infty$. (For our approach, the time parameter ~$d$ ~will turn out to correspond to the previously mentioned dimension parameter.) Our aim is to develop an algorithmic framework for the estimation of {\it change sets} ~$S$, where the means ~$E(X_k(i,j))$ ~for ~$(i,j)\in S$ ~differ from the means corresponding to  ~$(i,j)\not \in S$. The proposed algorithm is based on theoretical results in the aforementioned works of \citet{Bai2010},  \citet{vert2011a} and  \cite{torg2014v2}. The method is especially suitable for the estimation of multiple change sets, where their number does not necessarily has to be known in advance, and the approach may also be easily extended to other, more complex situations, e.g., with an irregular domain and irregular change sets. As a special case, the same principle may also be applied in a straightforward manner to the common multiple change point estimation within the panel data framework in \cite{torg2014v2} or \citet{vert2011a} in order to obtain consistent estimates for all changes when the number of panels tends to infinity.

For demonstration purposes, the suggested approaches are implemented with a graphical user interface as a Matlab application which can be obtained from the author or via \url{www.mi.uni-koeln.de/~ltorgovi}.

For some problems and approaches for change-set detection that are remotely related to our situation under consideration we refer the reader e.g. to \citet{Arnold2012}, \citet{Arnold2014} and the references therein (in particular to \citet{Spokoiny2000} and to the review article of \citet{Qui2007}).

\subsection*{Notation}
First, we need to introduce some notation in order to formulate our model. Consider a set ~$S\subset \mR^2$ ~of two-dimensional points, i.e. 
\[
	S=\big\{ (i,j) \;| \; i=i_1,\ldots, i_m\in \mN,\; j=j_1,\ldots, j_n\in \mN, \;m,n\in\mN \big\}.
\]
We call any point ~$(i,j\pm1), (i\pm1,j) \in S$ ~to be {\it adjacent } or a {\it neighbour } to ~$(i,j)\in S$. Correspondingly, two sets ~$S_1,S_2\subset S$ ~are called adjacent if at least two nodes ~$u\in S_1$, $v\in S_2$ ~exist that are adjacent to each other.  Furthermore, we will associate ~$S$ ~with an undirected graph such that each point ~$(i,j)\in S$ ~corresponds to a node and such that all adjacent nodes are connected by edges. The {\it boundary } of ~$S$ ~is a subset ~$B\subset S$ ~which contains only nodes ~$(i,j)\in S$ ~that have less than four distinct neighbours, i.e. nodes that are not {\it ~$4$-connected}. Correspondingly, an {\it interior } node  ~$(i,j)\in S$ ~has to have four neighbours within ~$S$. The set ~$S$ ~is called {\it connected } whenever the associated graph is connected, i.e., if there exists a path between any two nodes ~$n_1,n_2\in S$.

As usual, we define the distance ~$d(u,v)$ ~of two nodes ~$u, v \in S$ ~w.r.t. the set ~$S$ ~as the shortest path between them (within ~$S$). Accordingly, we define the distance of two sets ~$S_1, S_2\subset S$ ~w.r.t. the set ~$S$ ~as 
\[
	d(S_1,S_2)=\inf \{d(u,v)|\;u \in S_1, v\in S_2\},
\]
with ~$\inf\emptyset=\infty$ ~in which case the sets are obviously disjoint. The {\it Jaccard distance } between two sets ~$A,B\subset S$ ~is defined by
\begin{equation}\label{eq:jaccard}
	d_J(A,B) = \frac{|A\cup B |-|A\cap B|}{|A\cup B |}.
\end{equation}
This is a common measure of distance between two subsets, which will be used to quantify the precision of our estimates later on.
We are now in a position to formulate our actual model.

\subsection*{Statistical model}
We consider a spatio temporal {\it signal plus noise } model given by
\begin{equation}\label{eq:spatial_model}
		X_k(i,j) = m_k(i,j)  + \varepsilon_k(i,j)
\end{equation}
for ~$k=1,\ldots,d$ ~and ~$(i,j)\in D$, where ~$D$ ~is assumed to be a rectangular domain, i.e.
\[
	D=\{(i,j)|\; i=1,\ldots,m,\; j=1,\ldots,n\},
\]
with ~$m,n\geq 4$ ~being fixed integers. Here, the ~$m_k(i,j)$ ~are the deterministic signals and ~$\varepsilon_k(i,j)$ ~are the random variables representing the noise. One may interpret the sequence ~$\{X_k\}$ ~as a random field defined on the lattice ~$D$ ~and consider ~$k$ ~to be the time parameter.  Throughout, we assume the family of random variables
\[
	\{\varepsilon_k(i,j),\;i=1,\ldots,m,\;j=1,\ldots,n\}
\] to be i.i.d. for each ~$k$, identically distributed in ~$k$ with ~$0<E(\varepsilon_1(1,1)^2)=\sigma^2<\infty$ ~and centered. The data may be dependent in time, for which appropriate conditions will be imposed later on. Additionally, we have to assume uniformly bounded finite fourth moments.

For the general setting we may assume a partitioning of the domain ~$D$ ~as
\begin{equation}\label{eq:decomposition_domain}
	D = S_1 + \ldots + S_K,
\end{equation}
for some ~$K>1$, where each set ~$S_k$, $k=1,\ldots,K$ ~is non-empty and is assumed to be connected. Further, we assume piecewise constant means, i.e.
\begin{equation}\label{eq:mean_model}
	m_k(i,j) = \sum_{l=1}^K m_k(S_l)\one_{S_l}(i,j),
\end{equation}
with  ~$m_k(S_l)\in\mR$, such that for any ~$k$ ~it holds that ~$m_k(S_l)\neq m_k(S_u)$ ~for all adjacent sets ~$S_l,S_u$ ~with ~$u\neq l$. Our goal is the estimation of the partition ~$S_1,\ldots, S_K$ ~based on the noisy sequence ~$X_k$. (The Assumption \eqref{eq:mean_model} can be related e.g. to \citet[eq. (2)]{Spokoiny2000} but the statistical model there is non temporal.) As already mentioned, for ~$m=1$ ~the setting \eqref{eq:spatial_model} and \eqref{eq:mean_model} fits into the multiple change point scenario in panel data with fixed time parameter where the number of panels ~$d$ ~tends to infinity (cf., eg., \citet{vert2011a} and \citet{torg2014v2}).

The rectangular domain ~$D$ ~is chosen for simplicity of exposition only and as already mentioned all results discussed in this paper can be easily extended to more complex situations in a straightforward manner.  Recall that we consider asymptotics for ~$d\rightarrow\infty$.
\begin{definition} Assume the partitioning \eqref{eq:decomposition_domain}. We will call the sets ~$S_l$ ~for ~$l=1,\ldots,K$ ~to be {\it common change sets } if:
\begin{enumerate}
\item The sets ~$S_l$ ~are connected.
\item For all {\it total average changes}, defined as
 \begin{equation}\label{eq:delta_definition}
	\Delta^2_\infty(S_l,S_p) = \lim_{d\rightarrow \infty}\frac{\sum_{k=1}^d|\Delta_k|^2}{d},
\end{equation}
with ~$\Delta_k(S_l,S_p)=m_k(S_l)-m_k(S_p)$, it holds, for all adjacent sets ~$S_l,S_p$ ~with ~$l\neq p$, that 
 \begin{equation}\label{eq:delta_non_negative}
 0<\Delta^2_\infty(S_l,S_p)<\infty,
 \end{equation}
which quantifies the notion ``common'' in our setting.
\end{enumerate}
\end{definition}

Since the extension to multiple change sets is straightforward, we will (mostly) restrict ourselves in Sections \ref{sec:preliminaries}-\ref{sec:simulations} below to the single change set case, i.e.  to ~$K=2$, where ~$D = S + S^c$ ~and ~$S$, $S^c$ ~are both formally common change sets.  However, here it is more convenient to think that ~$S^c$ ~reflects the normal state region and ~$S$ ~is the only common change set differing from that normal state. We will use this terminology for brevity and simply write ~$\Delta^2_\infty=\Delta^2_\infty(S,S^c)$ ~in this situation.

\subsection*{Motivation}
The model \eqref{eq:spatial_model}-\eqref{eq:mean_model} with change sets ~$S_1,S_2,\ldots$ ~states a natural spatial extension of the setting considered e.g. by \cite{torg2014v2}.

One may think of digital imaging and assume a rectangular image sensor, i.e. an array of pixel sensors, corresponding to the domain ~$D$. Further, assume the image sensor to record a large test-sequence ~$X_1, \ldots, X_d$ ~of images that represent the light intensity (e.g. as monochrome grayscale images) and assume that the measurements, i.e. the images, are affected by some random noise.

Altogether, each image corresponds to an observation ~$X_k$ ~where ~$X_k(i,j)$ ~represents the measured intensity by the ~$(i,j)$-th pixel, where ~$m_k(i,j)$ ~is the true image intensity (cf., e.g., \citet{Qui2007}). Now, one may think of change sets ~$S_1,S_2\ldots$ ~to correspond e.g. to objects in the image that should be segmented or to a set of faulty pixels. 
The estimates discussed in this article can be used to estimate such sets based on a sufficiently long sequence of observations, i.e. when ~$d$ ~is large.

In contrast to more established settings and approaches (cf., e.g., \citet{Spokoiny2000} and \citet{Qui2007}) our aim here is to identify the partitioning based on a whole sequence of images. It is important that our model allows the means ~$m_k$ ~to change in each observation ~$k$ ~at any point ~$(i,j)$. For instance, we may think of changing lighting conditions while the ~$X_k$, $k=1,\ldots,d$ ~are recorded. (Otherwise, as shown in  Figure \ref{fig:averagingprobs}, assuming the means ~$m_k(i,j)$ ~to be constant across all ~$k$, one could e.g. simply rely on averages ~$\bar{X}_k(i,j)$ ~for each point ~$(i,j)$ ~to obtain a partitioning.) 
\\
\\
This article is organized as follows. We begin with preliminaries in Section \ref{sec:preliminaries}, where we briefly recall some theoretical results of \citet{vert2011a} and \cite{torg2014v2} on which our algorithms will be based. In Section \ref{sec:estimation_procedure} we describe the estimation algorithm together with the conditions that ensure consistency of the estimates. In Section \ref{sec:simulations} we show some simulation results to demonstrate the performance for finite ~$d$ ~and especially to show that even moderate ~$d$'s yield reasonable results.

\section{Preliminaries}\label{sec:preliminaries} 
Assume a ``single'' common change set scenario, i.e., observations ~$X_k$ ~in \eqref{eq:spatial_model} on a domain ~$D=S+S^c$ ~with common change sets ~$S$, $S^c$. Further, consider $(i,r)$-th horizontal {\it sub-slices }
\begin{equation}\label{eq:first_slice}
	\bY_j^{(i,r)}=[Y^{(i,r)}_{j,1},\ldots, Y^{(i,r)}_{j,d}]^T, \qquad j=1,\ldots,N,
\end{equation}
with ~$Y^{(i,r)}_{1,k},\ldots, Y^{(i,r)}_{N,k}$, obtained from these observations ~$X_k$ ~by setting
\[
	Y^{(i,r)}_{j,k} :=  X_k(i,r+j-1) 
\]
for each ~$k=1,\ldots,d$, ~$j=1,\ldots,N$ ~and some ~$r\in\{1,\ldots,n-N+1\}$ ~with ~$4\leq N\leq n$. Accordingly, we set ~$\eta^{(i,r)}_{j,k} :=  \varepsilon_k(i,r+j-1)$ ~for the corresponding innovations. The series ~$\bY_1,\ldots,\bY_N$ ~can be interpreted as panel data with ~$d$ ~panels and finite time horizon ~$N$. Notice, that the time parameter of the original spatial observations has now become a dimension parameter.

Assume that the particular ~$(i,r)$-th sub-slice ~$\{\bY^{(i,r)}_{j}\}_{j=1,\ldots,N}$ ~intersects a change set ~$S$ ~such that 
\begin{equation}\label{eq:condition_u}
(i,r+j-1) \in \begin{cases}S^c \;(S)& \text{for ~$j=1,\ldots,u(i,r)$,}\\
S \;\;(S^c)& \text{for ~$j=u(i,r)+1,\ldots,N$,}
\end{cases}
\end{equation}  
holds true for some ~$1\leq u(i,r)<N$~ and where  ~$r\in\{1,\ldots,n-N+1\}$. This situation is illustrated for the ~$(9,3)$-rd sub-slice of size ~$N=6$ ~and with ~$u(9,3)=2$ ~in Figure \ref{fig:schemealg} below. (We set formally ~$u(i,r)=\infty$ ~whenever \eqref{eq:condition_u} does not hold.)
\begin{figure}[H]%
\centering 
\includegraphics[width=0.75\textwidth]{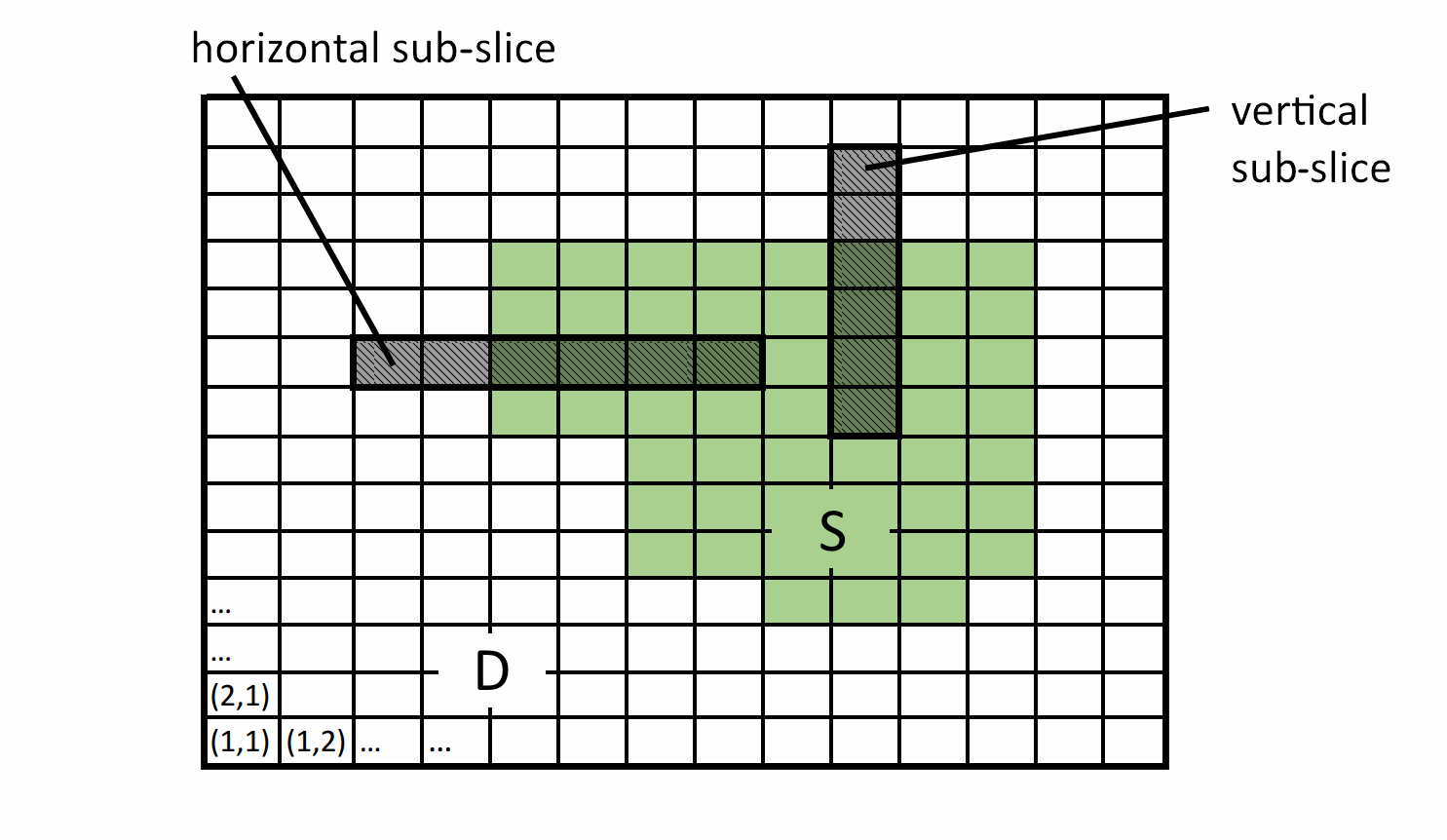}
\caption{The scheme represents the domain ~$D$ ~and sub-slices which intersect the change set ~$S$.}
\label{fig:schemealg}
\end{figure}
Altogether, we have a classical {\it change in the mean } scenario for any ~$\{\bY^{(i,r)}_{j}\}_{j=1,\ldots,N}$ ~where ~$u=u(i,r)$ ~corresponds to a single change point, i.e. \eqref{eq:first_slice}, fits into the framework considered in \cite{torg2014v2}. In order to estimate ~$u$ ~one may use a weighted CUSUM estimate 
\begin{equation}\label{eq:estimate}
	\hat{u}=\hat{u}(i,r)=\argmax_{p=1,\ldots,N-1} w(p,N)\sqrt{\sum_{k=1}^d\big|\sum_{j=1}^p(Y^{(i,r)}_{j,k}-\bar{Y}^{(i,r)}_{N,k})\big|^2}.
\end{equation}
Here, we restrict our considerations to a typical class of weighting functions, i.e.
\[
 w(p,N) = ((p/N)(1-p/N))^{-\gamma}
\]
parametrized by some ~$\gamma\in[0,1/2)$ ~which controls the sensitivity. The ~$\argmax$ ~in \eqref{eq:estimate} is defined, as usual, as the smallest index at which the maximum is attained and ~$\bar{Y}_{N,k}=\sum_{i=1}^N Y_{i,k}/N$.
\\
\\
To define classes of reasonable estimates for our original model \eqref{eq:spatial_model} we would like to make use of the sub-slices \eqref{eq:first_slice} and of the estimates \eqref{eq:estimate}, together with the corresponding theoretical results of \citet{vert2011a} and \cite{torg2014v2}. Therefore, we define the normalized {\it noise to change ratio } parameter ~$\rho$ ~w.r.t. the set ~$S$ ~and w.r.t. to the length of sub-slices ~$N$ ~as
\begin{equation}\label{eq:noisechangeratio}
	\rho(S)=\frac{1}{N}\frac{\sigma^2}{\Delta^2_\infty}.
\end{equation}
We also need the following assumption corresponding to \citet[Assumption (2.14)]{torg2014v2}:
\begin{equation}\label{eq:cross_var_spatial}
	 \frac{1}{d^2}\Var\left(\sum_{k=1}^d \Big(\sum_{j=1}^p (\eta_{j,k} - \bar{\eta}_{N,k})\Big)^2 \right)=o(1),
\end{equation}
as ~$d\rightarrow\infty$, for every ~$p=1,\ldots,N-1$. \eqref{eq:cross_var_spatial} is a weak dependence condition that e.g. clearly holds true if the ~$X_k$ ~are $m$-dependent (in particular independent) and identically distributed in ~$k$ (cf. \citet{torg2014v2}). Notice that we write ~$\eta_{j,k}$ ~instead of ~$\eta^{(i,r)}_{j,k}$ ~because \eqref{eq:cross_var_spatial} does not depend on the parameters ~$i$, $r$.
\\
\\
Our starting point is the fact that the above estimates \eqref{eq:estimate} consistently estimate ~$u$ ~under the assumptions on the model \eqref{eq:spatial_model} and under the Assumptions \eqref{eq:condition_u} and \eqref{eq:cross_var_spatial}. More precisely, we rely on the following ``switching''  behaviour of the estimates:
\vspace{10pt}
\begin{enumerate} 
\item If there is a single change point ~$u(i,r)$ ~w.r.t. the ~$(i,r)$-th sub-slice \eqref{eq:first_slice}, such that condition \eqref{eq:condition_u} is fulfilled, then ~$\hat{u}$ ~estimates the true change-point consistently for any ~$\gamma\in[0,1/2)$ ~if the noise to change ratio is below a positive threshold ~$R(\gamma,u,N)$. As ~$d\rightarrow\infty$, it holds that
\begin{equation}\label{eq:non-spurios}
	P(\hat{u}=u)\rightarrow 1
\end{equation}
for any ~$\gamma\in[0,1/2)$ ~given that ~$\rho<R(\gamma,u,N)$ (cf., eg., \citet[Theorem 2]{vert2011a} and \citet[Theorems 2.6 and 2.13]{torg2014v2}). The optimal threshold ~$R$ ~strongly depends on the parameter ~$\gamma$. The particular values for ~$R(0,u,N)$ may be obtained from \citet[Theorem 2]{vert2011a} in a closed form. Furthermore, it holds that ~$R(1/2,u,N) = \infty$ (cf., e.g., \citet[Theorem 3]{vert2011a} in the Gaussian i.i.d. case and take \citet[Theorem 2.6]{torg2014v2} into account regarding the nonparametric and dependent settings). We set
\[
	\cR(\gamma,N) :=\min_{u=1,\ldots,N-1} R(\gamma,u,N),
\]
which again does not depend on parameters ~$i$, $r$.
\item If there is no change point at all in the ~$(i,r)$-th sub-slice \eqref{eq:first_slice}, i.e. if it holds that ~$(i,r+j-1) \not \in S$ ~for ~$j=1,\ldots,N$ ~or that ~$(i,r+j-1) \in S$ ~for ~$j=1,\ldots,N$, then ~$\hat{u}$ ~estimates a spurious change. It holds that, as ~$d\rightarrow\infty$, 
\begin{equation}\label{eq:spurios}
	P\Big(\hat{u}=\lfloor N/2\rfloor \lor \hat{u}=\lceil N/2\rceil \Big)\rightarrow 1
\end{equation}
for any ~$\gamma\in[0,1/2)$ (cf. \citet[Remark 2.7]{torg2014v2}).
\end{enumerate}
\vspace{10pt}
We do not have closed form expressions for ~$R(\gamma)$ ~if ~$\gamma\in(0,1/2)$. However, from \citet[disp. (2.12)]{torg2014v2} it is clear that ~$R(\gamma)$ ~tends to infinity as ~$\gamma\uparrow1/2$ ~(cf. also further approximations to ~$R(\gamma)$ ~in \citet[Proposition 2.15]{torg2014v2}). 
\\
\\
The above switching behaviour in \eqref{eq:non-spurios} and in \eqref{eq:spurios} will provide consistent change set estimates ~$\hat{S}$ ~for ~$S$ ~in Section \ref{sec:estimation_procedure}.  Notice that the above switching property holds true for series \eqref{eq:first_slice} of any length ~$N\geq 4$, i.e. also for small single digit series of size ~$N\in\{4,\ldots,9\}$. In order to have a unique limit in \eqref{eq:spurios} we will consider only even ~$N$. Finally, we would like to mention that any other estimate ~$\hat{u}$ ~with an analogous switching behaviour might be used for scanning and aggregation in the next section as well.

\newpage
\section{Estimation procedure}\label{sec:estimation_procedure}
We stick to the single common change set scenario of the previous section. The key to the estimation of change sets in model \eqref{eq:spatial_model} will be a horizontal and/or a vertical overlapping scanning approach. The idea is to reduce the global problem of the change set estimation to many local single change point problems. This will allow us to lean on the results of \citet{vert2011a} and of \citet{torg2014v2} which were summarized in Section \ref{sec:preliminaries}.

We propose a four step procedure which is outlined in the following. Notice that each step 1-4 may be performed horizontally or vertically even though some steps are described for the horizontal approach only. Moreover, we explicitly allow to combine the horizontal with the vertical approach by proceeding consecutively. (The vertical approach proceeds in the very same manner with the obvious modifications. Clearly, the notation of the previous Section \ref{sec:preliminaries} has to be adapted as well which will also be indicated below.) 
\vspace{10pt}
\begin{enumerate}
	\item   {\bf Slicing:}
	\vspace{5pt}
	\begin{itemize}
	\item The time series ~$\{X_k\}$ ~is sliced into ~$m$ ~non-spatial ~$d$-dimensional time series ~$\{Y^{(i)}_{j,k},  k=1,\ldots,d\}_{j=1,\ldots,n}$ ~for ~$i=1,\ldots, m$ ~given by
	\begin{equation}\label{eq:global_slice}
		Y^{(i)}_{j,k} := X_k(i,j), \qquad j=1,\ldots,n, \quad k=1,\ldots, d.
	\end{equation}
	 We will denote ~$Y^{(i)}_{j,k}$  ~as a the ~$i$-th horizontal {\it slice } in the following. Similarly, we may define a ~$j$-th vertical slice by ~$\smash{\tilde{Y}^{(j)}_{i,k} := X_k(i,j)}$ ~for any ~$i=1,\ldots,m$, $k=1,\ldots, d$.
	  
	 \item  Now, we tacitly assume that ~$N$ ~is even and subdivide each ~$i$-th horizontal slice into ~$n-N$ ~overlapping sub-slices 
	 \begin{equation}\label{eq:subslice_def}
	\bY_j^{(i,r)}=[Y^{(i,r)}_{j,1},\ldots, Y^{(i,r)}_{j,d}]^T, \qquad j=1,\ldots,N
	\end{equation}
	 	 of size ~$N$, which are indicated by the parameter ~$r$, and where
	\begin{equation}\label{eq:subslice_def2}
		Y^{(i,r)}_{j,k}:= Y^{(i)}_{r+j-1,k},\qquad j=1,\ldots,N
	\end{equation}
	for ~$k=1,\ldots,d$, $r=1,\ldots,n-N+1$ ~and ~$i=1,\ldots, m$. Similarly, we may subdivide the ~$j$-th vertical slice into vertical sub-slices by
	\[
		\tilde{\bY}_i^{(j,r)}=[Y^{(j,r)}_{i,1},\ldots, Y^{(j,r)}_{i,d}]^T, \qquad i=1,\ldots,N, 
	\] 
	where ~$\smash{\tilde{Y}^{(j,r)}_{i,k}:=\tilde{Y}^{(j)}_{r+i-1,k}}$, $i=1,\ldots,N$, for ~$k=1,\ldots,d$, $r=1,\ldots,n-N+1$ ~and ~$j=1,\ldots, n$. Definition \eqref{eq:subslice_def} resembles \eqref{eq:first_slice}. However, in the 3rd step it will be convenient to think of all horizontal (vertical) sub-slices as parts of the same horizontal (vertical) slice, respectively.
	\end{itemize}
\end{enumerate}
\vspace{10pt}
\begin{enumerate}[resume]
	\item {\bf Scanning for critical points (Aggregation):} 
	\vspace{5pt}
	\begin{itemize}  
	\item Any sub-slice \eqref{eq:subslice_def}, or the vertical counterpart, is now treated as an individual time series to which we apply a single change-point estimate \eqref{eq:estimate} with any ~$\gamma\in[0,1/2)$ ~as described in Section \ref{sec:preliminaries}. 
	 (Also we tacitly assume the necessary modifications for vertical slices). Since we have ~$n-N$ ~sub-slices for each ~$i$, we aggregate ~$n-N$ ~estimated change point locations
	 \[
	 	\{\hat{u}(i,r)\in \mN, \quad r=1,\ldots,n-N+1\},
	 \]
	 again, for any ~$i$. These locations will be called {\it critical points } in our spatial context.
	\item The locations ~$\hat{u}(i,r)$ ~are integer-valued numbers since they are computed w.r.t. the ~$(i,r)$-th sub-slices. Hence, we need to map them back on our grid domain ~$D$ ~via
	\[
		\hat{U}(i,r) :=(i,\hat{u}(i,r) + r - 1) \in D
	\]
	for ~$r=1,\ldots,n-N+1$, $i=1,\ldots,m$. For theoretical reasons, we will restrict the admissible change sets ~$S$ ~by requiring ~$(i,j)\not\in S$ ~if ~$j>n-N+1$ ~or if ~$i>m-N+1$. Also, for technical reasons, we have to set ~$\hat{U}(i,r):=(i,0)$ ~for ~$r=n-N+2,\ldots,n$ ~and ~$i=1,\ldots,m$.
	\end{itemize}
	\vspace{10pt}
	\item {\bf Selection of relevant critical points:} \\\\
	In this step, our aim is to identify the boundary ~$B$ ~of the change set ~$S$. It will induce an estimate ~$\hat{S}$ ~in a straightforward manner. Observe that only those critical points ~$\hat{U}(i,r)$ ~which are adjacent to the change set ~$S$ ~(or lie in ~$B$) may help us to identify this boundary and therefore the set ~$S$. Asymptotically, i.e. as ~$d\rightarrow\infty$ ~and with probability tending to ~$1$, the points ~$v\in B$ ~will correspond to those ~$\hat{u}(i,r)$ ~and ~$\hat{U}(i,r)$ ~that are based on correct estimation \eqref{eq:non-spurios} and not on the spurious ones as in \eqref{eq:spurios}. Hence, we have to filter out the latter  by selecting a set ~$G$ ~of {\it relevant } critical points, based on ~$\hat{U}$, that is expected to be informative, based on suitable decision rules. 
\\
\\
We present the {\it overlapping }$(N,Q)$ rules, in form of a pseudocode. Let ~$H(i)$ ~denote the set of relevant points w.r.t. the ~$i$-th horizontal slice and recall that we assume ~$N\geq 4$ ~to be even.	The overlapping $(N,Q)$ rule is:
				\vspace{10pt} 
\begin{algorithmic} [1]
\STATE Choose some integer ~$1 \leq Q\leq N-2$
\FOR{$i=1$ ~to ~$m$} 
\STATE $H(i) \leftarrow \emptyset$
\FOR{$r=1$ ~to ~$n-N+1$} 
\IF{$\hat{U}(i,r)=\hat{U}(i,r+1)=\ldots = \hat{U}(i,r+Q)$} 
\STATE  $H(i) \leftarrow H(i)\cup \{\hat{U}(i,r)\}$
\ENDIF
\ENDFOR
\ENDFOR
\end{algorithmic}
\vspace{10pt}
The idea behind this algorithm is according to \eqref{eq:non-spurios} and \eqref{eq:spurios} that, assuming that the noise to change ratio lies below the threshold ~$\cR(\gamma,N)$, only the following cases may occur for the estimate ~$\hat{u}$ ~applied to \eqref{eq:subslice_def}:
	\vspace{10pt}
	\begin{enumerate}
		\item There is a single change at ~$1\leq u< N-1$. In that case we know that asymptotically, as ~$d\rightarrow\infty$, this point is estimated correctly as a critical point with probability tending to ~$1$. 
		\item There is more than one change point. This case will be excluded from our consideration (cf. conditions on the change sets \eqref{eq:restriction_ver} in Theorem \ref{eq:main_theorem} below).
		\item There is no change in this sub-slice. Hence, asymptotically as ~$d\rightarrow\infty$, we estimate ~$\hat{u}=\lfloor N/2\rfloor$ ~spuriously with probability tending to ~$1$.
	\end{enumerate}
	\vspace{10pt}

	For simplicity assume that ~$Q=1$ (the case ~$Q>1$ ~works in the same way). If some consecutive sub-slices, e.g. the ~$(i,r)$-th~ and ~$(i,r+1)$-th, intersect the change set region ~$S$, such that both have a single change point, i.e at ~$u(i,r)$ ~and at ~$u(i,r+1)$, then we are in case a) for both sub-slices and therefore ~$P(\hat{U}(i,r+1)=\hat{U}(i,r))\rightarrow 1$, as ~$d\rightarrow\infty$ which means that the condition of the 5th line, in the above algorithm, is fulfilled for ~$Q=1$. On the other hand, if there is no change in at least one of the two subslices, we have ~$P(\hat{U}(i,r+1)=\hat{U}(i,r))\rightarrow 0$, as ~$d\rightarrow\infty$, and the 5th line is always violated for any ~$Q$ . Hence, the sets ~$H(i)$, $V(j)$ ~will asymptotically contain only points that correspond to change-points in the sub-slices.
	 	
		\vspace{5pt}	
		The parameters ~$N$ ~and ~$Q$, allow us to control the sensitivity. In particular, a smaller ~$Q$ ~is less restrictive and therefore more sensitive, but less reliable.  The overlapping rule is sketched in Figures \ref{fig:change_profile} and \ref{fig:estimation_chunks} below. The illustration is based on a fragment of rectangular spatial observations of size ~$m=14$, $n=20$. Each field corresponds to a point ~$(i,j)$ ~in a straightforward manner. The sensitivity w.r.t. the parameters ~$(N,Q)$ ~and ~$\gamma$ ~is demonstrated in Figures \ref{fig:sensitivity1}-\ref{fig:sensitivity3}, below.
		
	\vspace{5pt}	 
	Subsequently, we write ~$V(j)$ ~for the set of relevant change points w.r.t. the ~$j$-th vertical slice. In case that we perform horizontal scanning only, we set formally ~$V(j)=\emptyset$ ~for ~$j=1,\ldots,n$ ~and analogously we set ~$H(i)=\emptyset$ ~for ~$i=1,\ldots,m$ ~if we would perform vertical scanning only.
	
	\vspace{5pt}
	The pooled set of all relevant critical points will be denoted by
	\[
		G = H(1)\cup \ldots \cup H(m) \cup V(1) \cup \ldots \cup V(n).
	\]
	\vspace{3pt}
	\item  {\bf Connecting relevant critical points:}  \\\\
	The set ~$G$ ~should asymptotically contain only nodes adjacent to the boundary ~$B$ ~of ~$S$. Hence, based on ~$G$, or on ~$H(i)$ ~and ~$V(j)$, we try to identify as many nodes of ~$S$ ~as possible.  Here, this is carried out for each slice separately. (We tacitly restrict the class of change-sets ~$S$ ~according to the Theorem \ref{eq:main_theorem}). Let ~$\hat{S}$ ~denote the estimate for ~$S$.
The horizontal procedure is:
\vspace{10pt}
\begin{algorithmic}[1] 
\STATE $\hat{S} \leftarrow \emptyset$
\FOR{$i=1$ to $m$} 
\IF{$H(i)=\{(i,x_1), (i,x_2), \ldots, (i,x_p)\}$, $x_1<x_2<\ldots <x_p$ ~with ~$p\geq 2$} 
\STATE $\hat{S}:=\hat{S}\cup \{(i,x_1+1),\ldots, (i,x_p)\}$
\ENDIF
\ENDFOR
\end{algorithmic}
\vspace{10pt}
and the analogous vertical procedure is:
\vspace{10pt}
\begin{algorithmic}[1]
\FOR{$j=1$ to $n$} 
\IF{$V(j)=\{(x_1,j), (x_2,j), \ldots, (x_p,j)\}$, $x_1<x_2<\ldots <x_p$ with $p\geq 2$} 
\STATE $\hat{S}:=\hat{S}\cup \{(x_1+1,j),\ldots, (x_p,j)\}$
\ENDIF
\ENDFOR
\end{algorithmic}
\vspace{10pt}
	\end{enumerate} 
	
\begin{figure}[H]%
\centering 
\vspace{-50pt}
\includegraphics[width=0.65\textwidth]{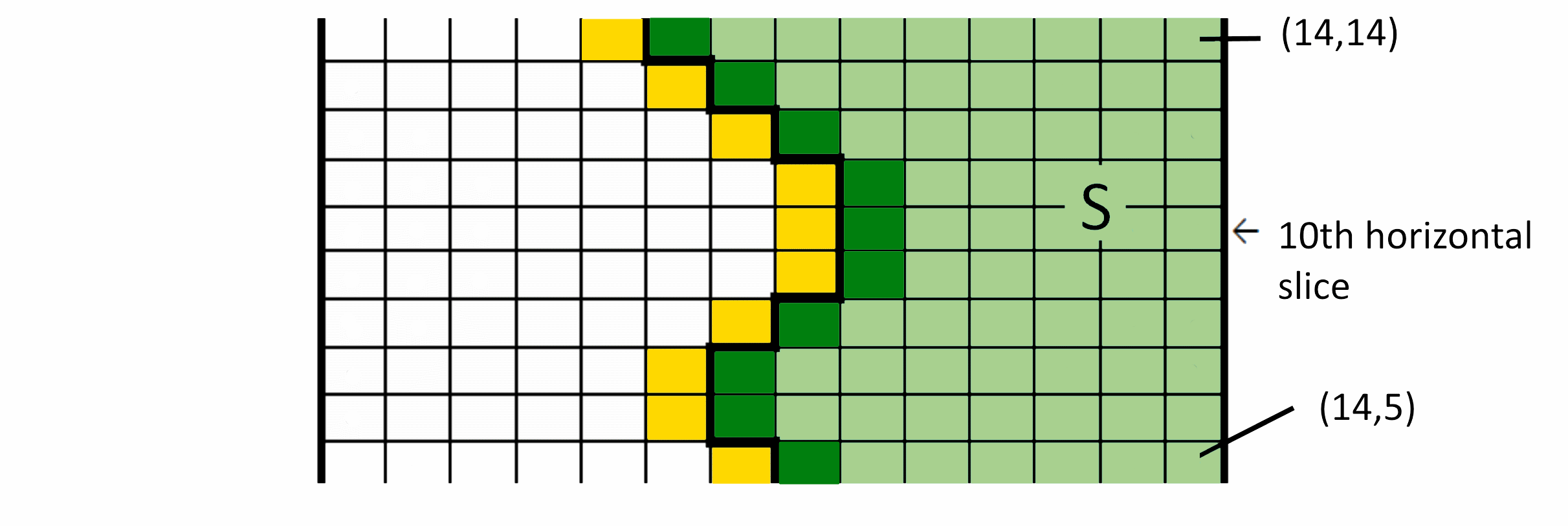}
\caption{The common change set ~$S$ ~is indicated by green color. The darker fields indicate the boundary of ~$S$. The yellow fields indicate relevant critical points that will be eventually selected, i.e. with probability tending to 1 as ~$d\rightarrow\infty$, applying the overlapping $(4,2)$ rule with any ~$\gamma\in[0,1/2)$.}
\label{fig:change_profile}
\end{figure}

\begin{figure}[H]%
\centering 
\subfloat[][scanning result for ~$N=6$ ~with any ~$\gamma<1/2$; step 2]{%
\includegraphics[width=0.35\textwidth]{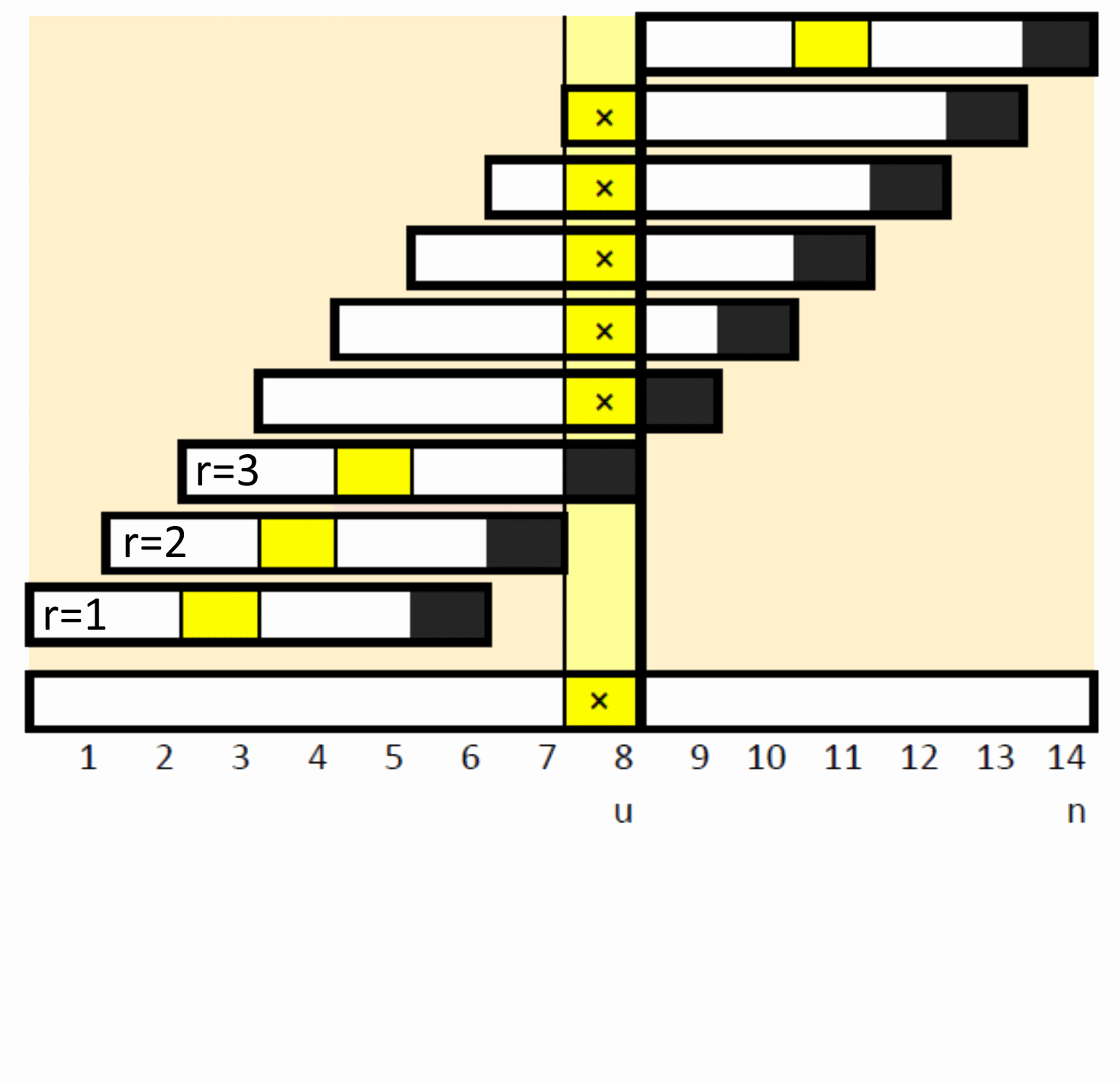}
}%
\subfloat[][overlapping (6,4) rule with any ~$\gamma<1/2$; step 3]{%
\includegraphics[width=0.35\textwidth]{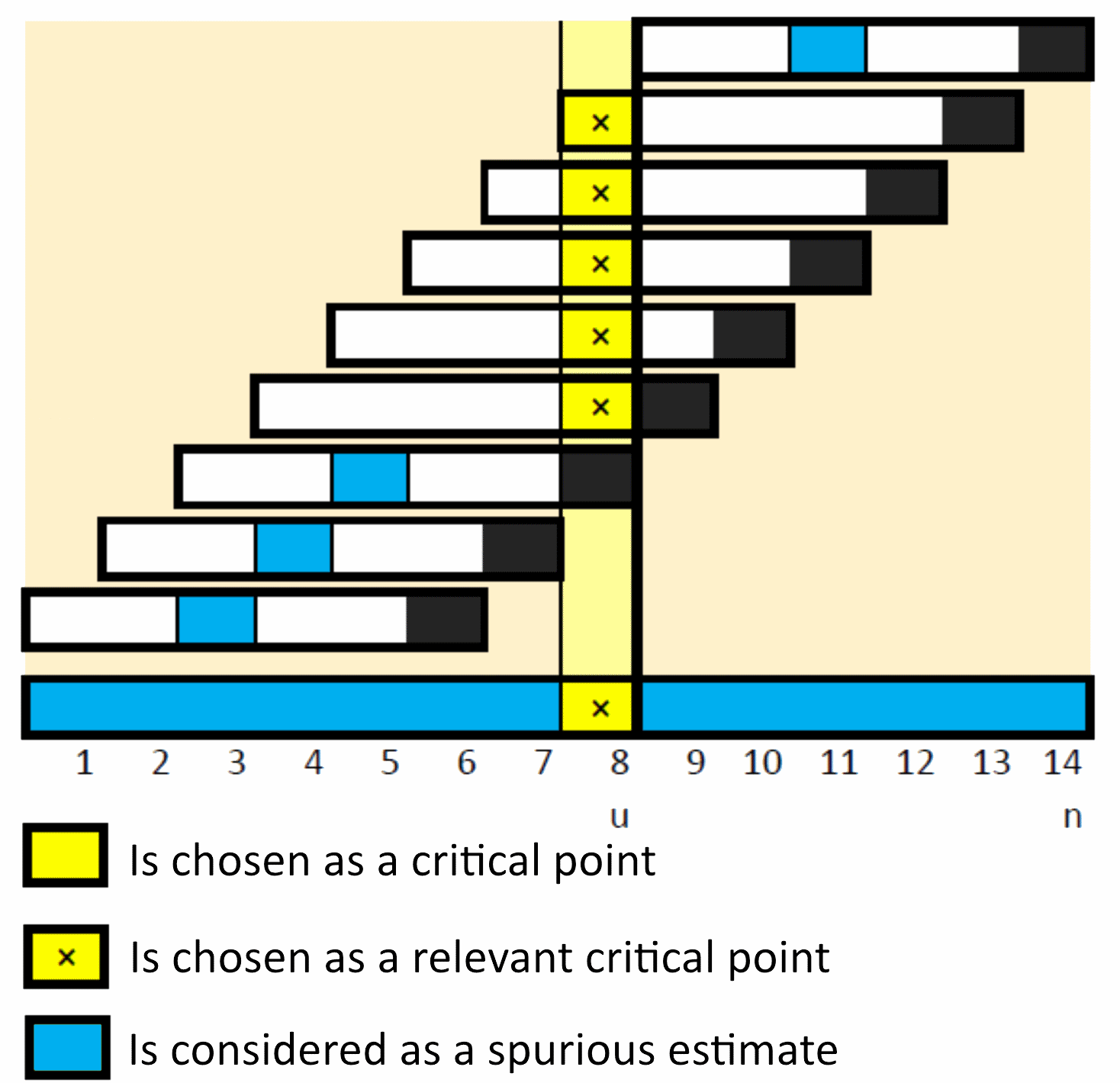}
}
\caption{The figures illustrate the algorithm for the overlapping rule. All figures show the ~$10$-th horizontal slice of Figure \ref{fig:change_profile} and the corresponding ~$n-N$ ~overlapping sub-slices. It is indicated which points are selected as relevant critical points asymptotically with probability tending to 1 as ~$d\rightarrow\infty$)}%
\label{fig:estimation_chunks}
\end{figure}

\begin{theorem}\label{eq:main_theorem}
Assume a rectangular domain ~$D$ ~with ~$\min\{m,n\}\geq \xi \geq 4$, a boundary ~$B$ ~and a common connected change set ~$S$ ~with ~$d(S,B)\geq \xi-1$ ~for some ~$\xi\in\mN$. Define vertical and horizontal intersections by 
\begin{align}\label{eq:restriction_ver}
\begin{split}
	H_i&:=S\cap \{(i,l)|\quad l=1,\ldots,n\},\\
	V_j&:=S\cap \{(l,j)|\quad l=1,\ldots,m\},
\end{split}
\end{align} 
for all  ~$j=1,\ldots,n$, $i=1,\ldots,m$, respectively. Furthermore, Assume that all ~$H_i$ ~and ~$V_j$ ~are either empty or connected sets. We have to state three different assumptions:
\begin{enumerate}
\item Assume that we use the horizontal approach and that for all  ~$i=1,\ldots,m$ ~it holds that ~$|H_i|\geq \xi$ ~if ~$H_i\neq\emptyset$.
\item Assume that we use the vertical approach and that for all  ~$j=1,\ldots,n$ ~it holds that ~$|V_j|\geq \xi$ ~if ~$V_j\neq\emptyset$.
\item Assume that we combine the horizontal together with the vertical approach and that for all  ~$i=1,\ldots,m$, $j=1,\ldots,n$ ~it holds that ~$\max\{|H_i|, |V_j|\}\geq \xi$ ~if ~$(i,j)\in H_i\cap V_j$.
\end{enumerate}
Assume that we use the overlapping (N,Q) rule, with ~$4\leq N\leq \xi$ ~and ~$1\leq Q\leq N-2$ ~where ~$N$ ~is even. Furthermore, let ~$\gamma\in[0,1/2)$ ~and the ratio ~$\rho$ ~be below  ~$\cR(\gamma,\xi)$. Under either of the above Assumptions 1-3, given that \eqref{eq:cross_var_spatial} holds true w.r.t. all sub-slices \eqref{eq:first_slice}, it holds that, as ~$d\rightarrow\infty$,
\begin{equation}\label{eq:consistent_set}
  P(\hat{S}=S)\rightarrow 1.
\end{equation}
\end{theorem}

\begin{proof}
Conditions on ~$H_i$ ~and ~$V_j$ ~ensure that only two cases may occur. Either a sub-slice contains a single change point or does not contain a change point at all.
The number of sub-slices is fixed and finite. Hence, the overall consistency of ~$\hat{S}$ ~follows from the consistency of all estimates ~$\hat{u}(i,r)$ ~in case of a change and from the fact of spurious estimation when there is no change (cf. \citet[Theorems 2.6, 2.13 and Remark 2.17]{torg2014v2}).
\end{proof}

\section{Simulations}\label{sec:simulations} 
We start this section by illustrating the estimation procedure and the corresponding Theorem \ref{eq:main_theorem} of the previous Section \ref{sec:estimation_procedure}.

We begin with parameters that will be common in our simulations. For simplicity, we consider a domain ~$D$ ~with ~$m,n=100$ ~and assume the noise to be i.i.d. normally distributed with ~$\varepsilon_1(1,1)=N(0,\sigma^2)$. We consider a single common change set scenario (see Section \ref{sec:introduction} and \ref{sec:preliminaries}) and define test change sets ~$S$~ with radius ~$w$, centered at a point $v\in D$, by
\[
	S_{w,v}:=\{u\in D | \; \|u-v\|_p \leq w\}.
\]
Here, $\|u-v\|_p$~ is the usual ~$p$-Norm for vectors in ~$\mR^2$ ~and ~$p=\infty$ ~denotes the maximum norm.
We call such sets {\it rectangular-shaped } for ~$p=\infty$, {\it round-shaped } for ~$p=2$ ~and {\it diamond-shaped } for ~$p=1$. The reference mean level ~$m_k(S^c)$ ~is set to ~$m_k=k$ ~for ~$k=1,\ldots,d$. 
\vspace{5pt}
\begin{remark}
Recall, that ~$d_J$ ~denotes the Jaccard distance defined in \eqref{eq:jaccard}. Clearly, relation \eqref{eq:consistent_set} implies ~$P(d_J(\hat{S},S)=0)\rightarrow 1$ ~as ~$d\rightarrow\infty$ ~which in turn yields ~$E(d_J(\hat{S},S))\rightarrow 0$  ~as ~$d\rightarrow\infty$. The latter follows e.g. due to uniform integrability of ~$d_J(\hat{S},S)\in[0,1]$.  In our simulations we demonstrate the influence of various parameters on the expected Jaccard distance ~$Ed_J=E(d_J(\hat{S},S))$ ~which is approximated based on ~$100$ ~repetitions. 
\end{remark} 
\vspace{5pt}
Table \ref{table:overlapping} shows ~$Ed_J$ ~for the overlapping ~$(N,Q)$ ~rules w.r.t. different parameters ~$d$, $N$, $Q$  ~and ~$\gamma$. Generally, it is not clear which combination of sensitivity parameters ~$(N,Q)$  ~and ~$\gamma$ ~is preferable. Hence, our advise is to plot different estimates and to rely on visual inspection (cf. Figures \ref{fig:sensitivity1} - \ref{fig:sensitivity3} below). Nevertheless, we see two tendencies where either the expected distance ~$Ed_J$ ~improves for larger ~$d$, e.g. for ~$N=4$, $Q=1$ ~and ~$\gamma=0.3$, or worsens, e.g. for ~$N=4$, $Q=1$ ~and ~$\gamma=0$. In accordance with the theory, the former happens if the ratio ~$\rho$ ~is below the threshold ~$\cR$ ~and the latter when ~$\rho$ ~is above. Notice, that the precision does not monotonously increase in ~$\gamma$. 

The Figures \ref{fig:sensitivity1} - \ref{fig:sensitivity3} are based on a spatio-temporal sequence which is illustrated in Figure \ref{fig:averagingprobs}. Comparing the Figures \ref{fig:sensitivity1} and \ref{fig:sensitivity2} for ~$\gamma=1/4$ ~we see that a larger ~$d$ ~improves the estimation. Clearly, a smaller ~$Q$ ~yields more sensitive estimates but on the other hand larger parameters ~$Q$ ~may isolate the change sets better. Table \ref{table:overlapping} and Figures \ref{fig:sensitivity1}-\ref{fig:sensitivity3} show that the usage of the horizontal procedure together with the vertical procedure might be better or worse than the plain horizontal approach. However, for some change sets, e.g. diamond-shaped, it is necessary to use both directions in order to obtain a consistent estimate  (cf. Figure \ref{fig:round_shaped_vertical}).

\begin{table}[H]
\begin{center}
\small
\begin{tabular}{l|rrrrr} 
\multicolumn{1}{l}{}&\multicolumn{1}{c}{$\gamma=0$}&\multicolumn{1}{c}{$\gamma=0.1$}&\multicolumn{1}{c}{$\gamma=0.2$}&\multicolumn{1}{c}{$\gamma=0.3$}&\multicolumn{1}{c}{$\gamma=0.4$} \tabularnewline
\hline
\hline
{\bf overlapping (4,1)} & & & & &  \tabularnewline 
$d=100$&0.46 (0.51)&0.47 (0.54)&0.49 (0.54)&0.50 (0.55)&0.51 (0.55)\tabularnewline
$d=200$&0.68 (0.52)&0.44 (0.46)&0.45 (0.53)&0.49 (0.55)&0.51 (0.55)\tabularnewline
$d=300$&0.89 (0.79)&0.52 (0.38)&0.41 (0.49)&0.48 (0.54)&0.51 (0.55)\tabularnewline
$d=500$&0.97 (0.94)&0.72 (0.52)&0.30 (0.28)&0.43 (0.53)&0.50 (0.55)\tabularnewline
$d=1000$&0.99 (0.99)&0.82 (0.66)&0.24 (0.06)&0.23 (0.35)&0.49 (0.54)\tabularnewline
\hline
\hline
{\bf overlapping (4,2)} & & & & & \tabularnewline 
$d=100$&0.92 (0.85)&0.74 (0.59)&0.53 (0.47)&0.44 (0.50)&0.45 (0.53)\tabularnewline
$d=200$&0.99 (0.98)&0.92 (0.85)&0.66 (0.48)&0.42 (0.43)&0.42 (0.52)\tabularnewline
$d=300$&1.00 (0.99)&0.95 (0.91)&0.73 (0.54)&0.38 (0.33)&0.41 (0.50)\tabularnewline
$d=500$&1.00 (1.00)&0.97 (0.94)&0.75 (0.56)&0.28 (0.16)&0.36 (0.47)\tabularnewline
$d=1000$&1.00 (1.00)&0.98 (0.97)&0.67 (0.45)&0.12 (0.02)&0.23 (0.35)\tabularnewline
\hline
\hline
{\bf overlapping (6,2)} & & & & & \tabularnewline 
$d=100$&0.42 (0.48)&0.42 (0.51)&0.43 (0.52)&0.45 (0.53)&0.46 (0.54)\tabularnewline
$d=200$&0.33 (0.36)&0.36 (0.45)&0.40 (0.50)&0.43 (0.53)&0.46 (0.54)\tabularnewline
$d=300$&0.24 (0.20)&0.27 (0.36)&0.36 (0.48)&0.41 (0.52)&0.45 (0.54)\tabularnewline
$d=500$&0.10 (0.04)&0.11 (0.16)&0.26 (0.38)&0.38 (0.50)&0.44 (0.53)\tabularnewline
$d=1000$&0.01 (0.00)&0.01 (0.01)&0.07 (0.12)&0.28 (0.41)&0.42 (0.52)\tabularnewline
\hline
\hline
{\bf overlapping (6,4)} & & & & & \tabularnewline 
$d=100$&1.00 (1.00)&1.00 (1.00)&1.00 (1.00)&0.95 (0.89)&0.65 (0.46)\tabularnewline
$d=200$&1.00 (1.00)&1.00 (1.00)&1.00 (1.00)&0.95 (0.90)&0.53 (0.30)\tabularnewline
$d=300$&1.00 (1.00)&1.00 (1.00)&1.00 (1.00)&0.94 (0.89)&0.40 (0.18)\tabularnewline
$d=500$&1.00 (1.00)&1.00 (1.00)&1.00 (1.00)&0.93 (0.86)&0.23 (0.06)\tabularnewline
$d=1000$&1.00 (1.00)&1.00 (1.00)&1.00 (1.00)&0.90 (0.82)&0.05 (0.00)\tabularnewline 
\end{tabular}\end{center}
\caption{The expected Jaccard distance ~$Ed_J$ ~for the overlapping algorithm for the horizontal (horizontal + vertical) approach. For simplicity we set ~$m_k(S)=k+(-1)^k$, $k=1,\ldots,d$, i.e. $m_k(S)=m_k(S^c)+(-1)^k$, and ~$\sigma^2=2$. The change set ~$S_{w,v}$ ~is rectangular-shaped centered at ~$v=(50,50)$ ~with ~$w=100/3$}\label{table:overlapping}
\end{table}

\begin{figure}[H]%
\centering 
\vspace{-20pt}
\subfloat[][horizontal approach]{%
\includegraphics[width=0.3\textwidth]{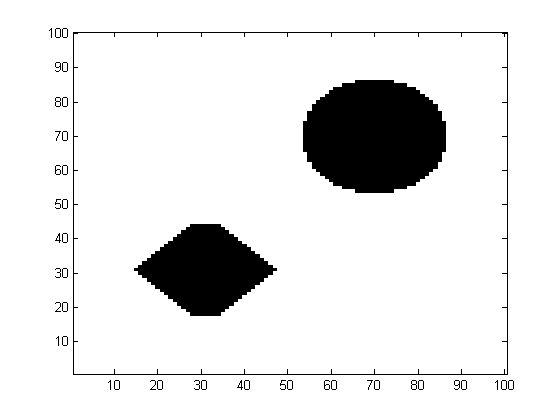}
}%
\subfloat[][horizontal+vertical]{%
\includegraphics[width=0.3\textwidth]{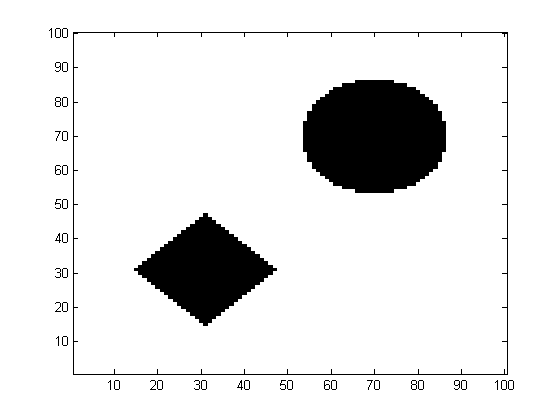}
}%
\caption{The figure shows the overlapping $(16,8)$ ~estimate with ~$\gamma=0$ ~for multiple diamond- and round-shaped change sets ~$D=S_D+S_R+S^c$ ~with ~$w=100/6$. The horizontal procedure is consistent for the round-shaped set but only the horizontal + vertical approach yields a consistent estimate for the diamond-shaped change set. It holds ~$m_k(S_D)=m_k(S^c)+(-1)^k$, $m_k(S_R)=m_k(S^c)$, $k=1,\ldots,d$ ~and ~$d=1000$ ~with ~$\sigma^2=1$}%
\label{fig:round_shaped_vertical}%
\end{figure}

\begin{figure}[H]%
\captionsetup[subfigure]{labelformat=empty, textfont=footnotesize}
\centering 
\subfloat[][]{\includegraphics[width=0.3\textwidth]{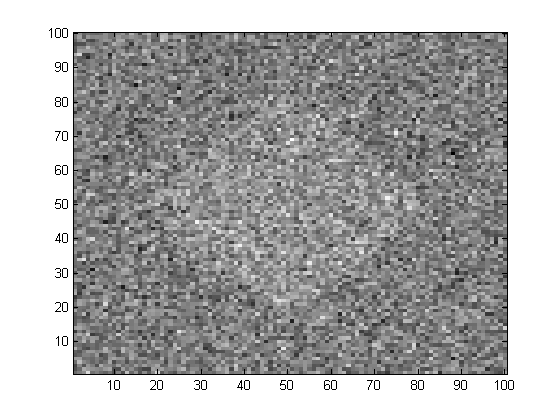}}
\subfloat[][]{\includegraphics[width=0.3\textwidth]{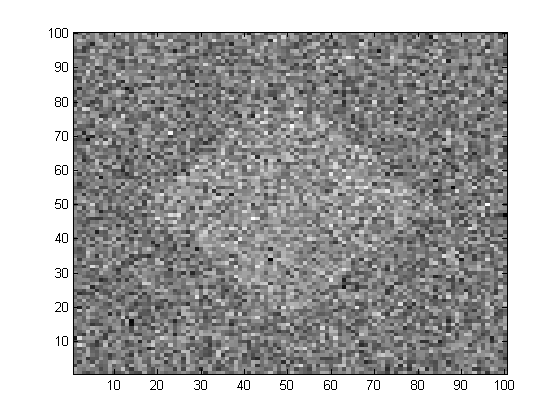}}
\subfloat[][]{\includegraphics[width=0.3\textwidth]{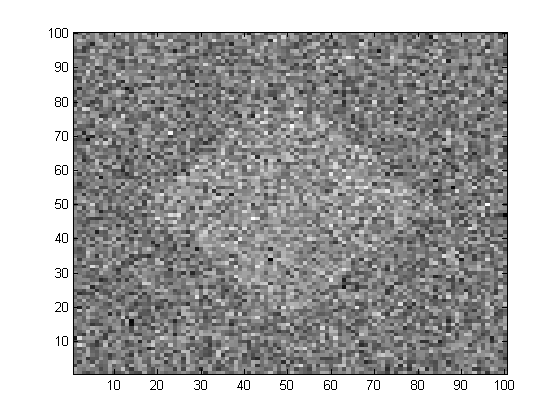}}
\\
\subfloat[][]{\includegraphics[width=0.3\textwidth]{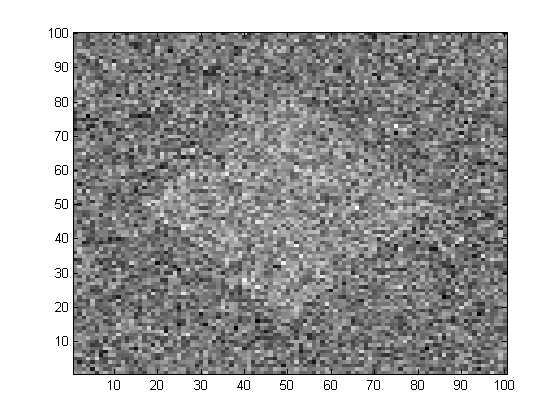}}
\subfloat[][]{\includegraphics[width=0.3\textwidth]{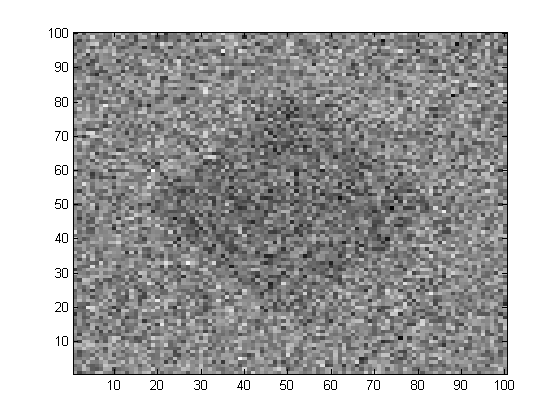}}
\subfloat[][]{\includegraphics[width=0.3\textwidth]{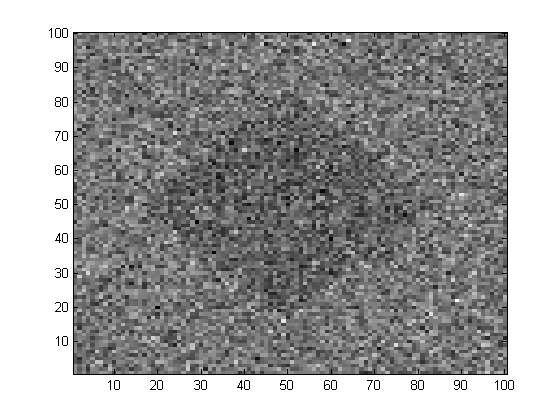}}
\\
\subfloat[][a)]{\includegraphics[width=0.3\textwidth]{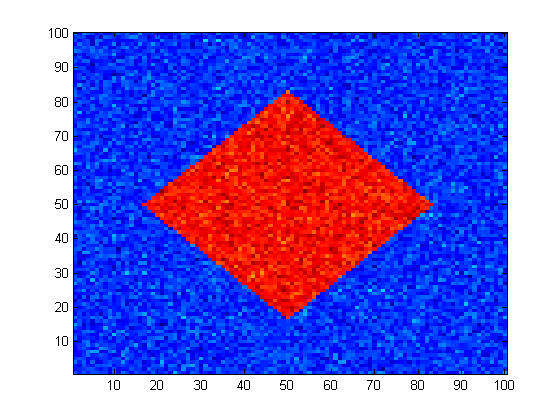}}
\subfloat[][b)]{\includegraphics[width=0.3\textwidth]{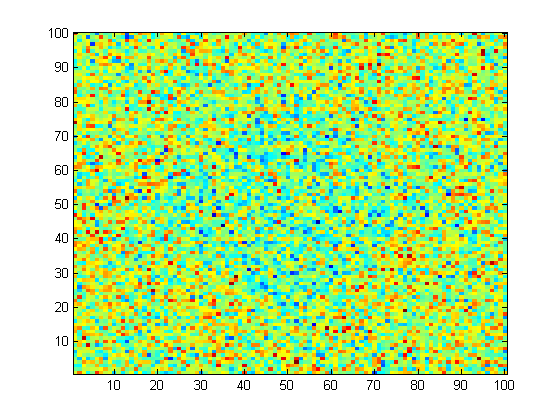}}
\subfloat[][c)]{\includegraphics[width=0.3\textwidth]{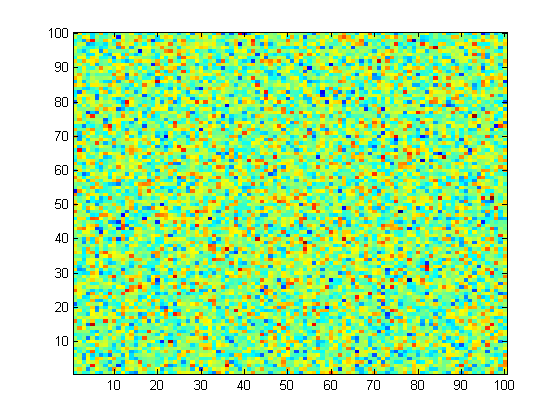}}
\caption{The first two rows show the observations ~$X_1,\ldots,X_6$ ~with ~$\sigma^2=2$ ~for ~$m_k(S)=k+(-1)^k$ ~and ~$m_k(S^c)=k$. The third row shows the averages ~$\bar{X}_{500}$. In a) for the simple case of ~$m_k(S)=0$ ~and ~$m_k(S^c)=1$. In b) for the case of ~$m_k(S)=0$ ~and ~$m_k(S^c)=(-1)^k$ and finally in c) for ~$m_k(S)=k$ ~and ~$m_k(S^c)=k+(-1)^k$. Notice that by averaging we loose (or at least do not gain) information in the last two settings}%
\label{fig:averagingprobs}
\end{figure}

\newpage
\begin{figure}[H]%
\captionsetup[subfigure]{labelformat=empty, textfont=footnotesize} 
\centering  
\vspace{-80pt}
\subfloat[][h, $(4,1)$,\\$\gamma=0$, rel.]{%
\includegraphics[width=0.25\textwidth]{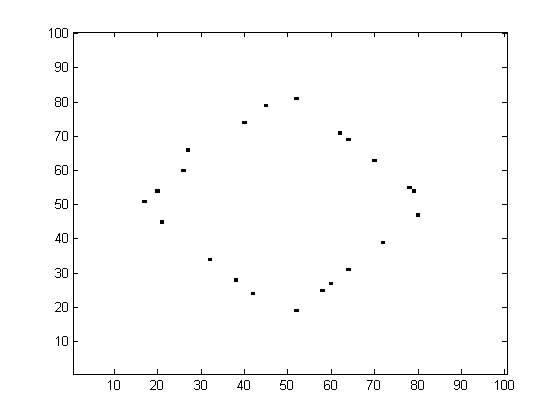}
}%
\subfloat[][h, $(4,2)$,\\$\gamma=0$, rel.]{%
\includegraphics[width=0.25\textwidth]{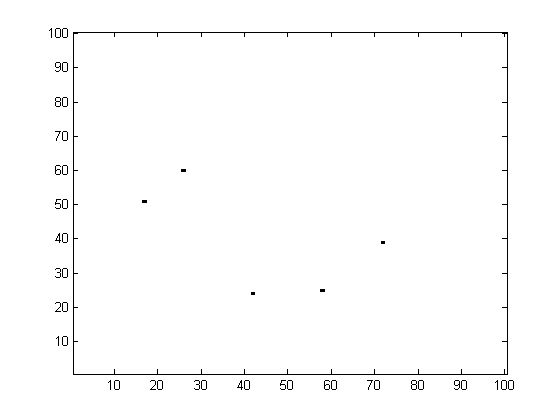}
}%
\subfloat[][h, $(4,1)$,\\$\gamma=0$, est.]{%
\includegraphics[width=0.25\textwidth]{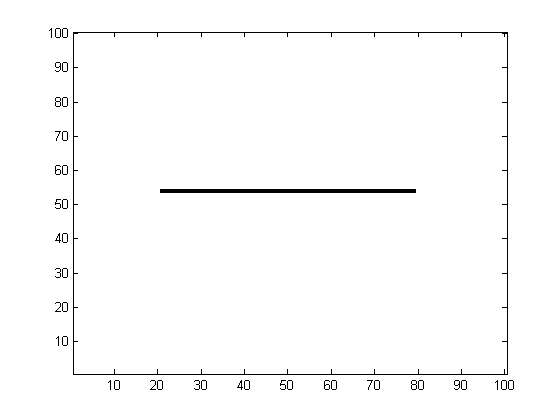}
}%
\subfloat[][h, $(4,2)$,\\$\gamma=0$, est.]{%
\includegraphics[width=0.25\textwidth]{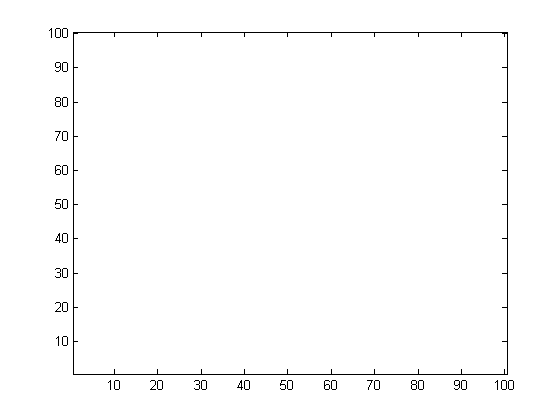}
}%
\\\vspace{-10pt}
\subfloat[][h, $(4,1)$,\\$\gamma=0.25$, rel.]{%
\includegraphics[width=0.25\textwidth]{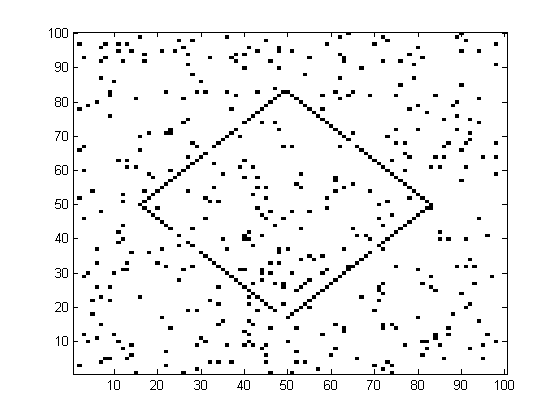}
}%
\subfloat[][h, $(4,2)$,\\$\gamma=0.25$, rel.]{%
\includegraphics[width=0.25\textwidth]{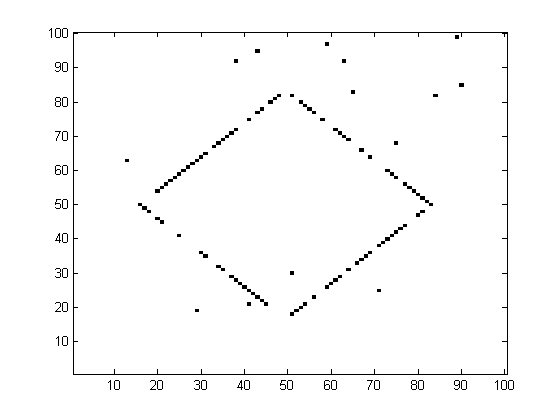}
}%
\subfloat[][h, $(4,1)$,\\$\gamma=0.25$, est.]{%
\includegraphics[width=0.25\textwidth]{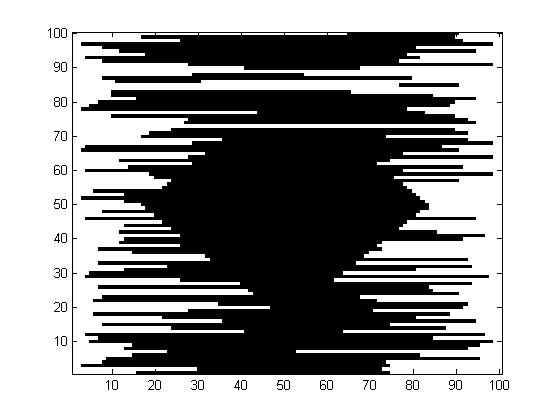}
}%
\subfloat[][h, $(4,2)$,\\$\gamma=0.25$, est.]{%
\includegraphics[width=0.25\textwidth]{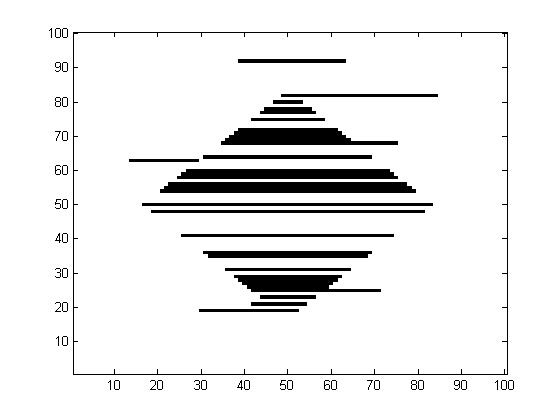}
}%
\\\vspace{-10pt}
\subfloat[][h, $(4,1)$,\\$\gamma=0.49$, rel.]{%
\includegraphics[width=0.25\textwidth]{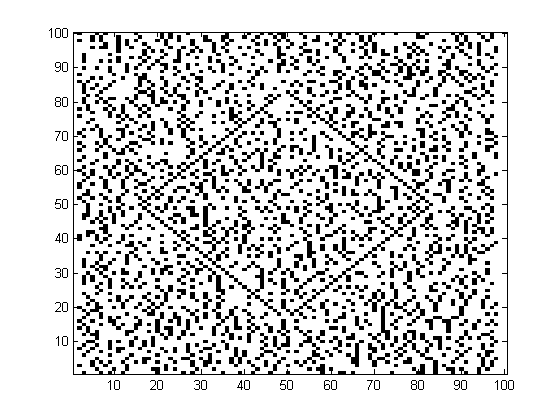}
}%
\subfloat[][h, $(4,2)$,\\$\gamma=0.49$, rel.]{%
\includegraphics[width=0.25\textwidth]{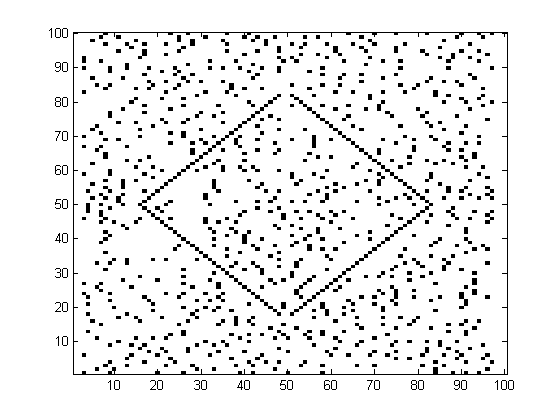}
}%
\subfloat[][h, $(4,1)$,\\$\gamma=0.49$, est.]{%
\includegraphics[width=0.25\textwidth]{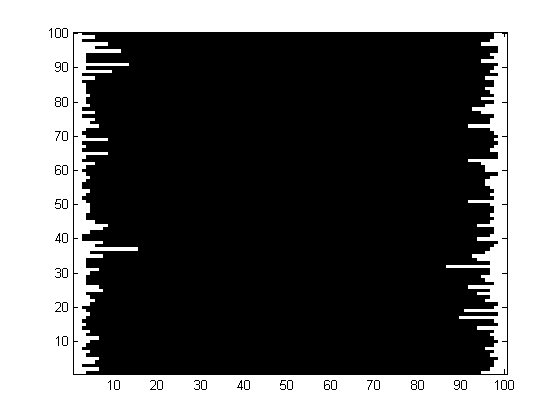}
}%
\subfloat[][h, $(4,2)$,\\$\gamma=0.49$, est.]{%
\includegraphics[width=0.25\textwidth]{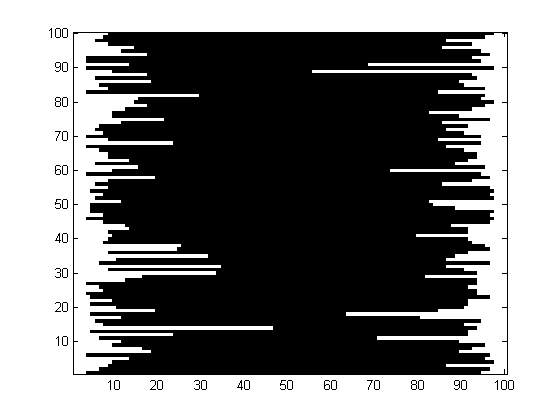}
}%
\\\vspace{10pt}
\subfloat[][h+v, $(4,1)$,\\$\gamma=0$, rel.]{%
\includegraphics[width=0.25\textwidth]{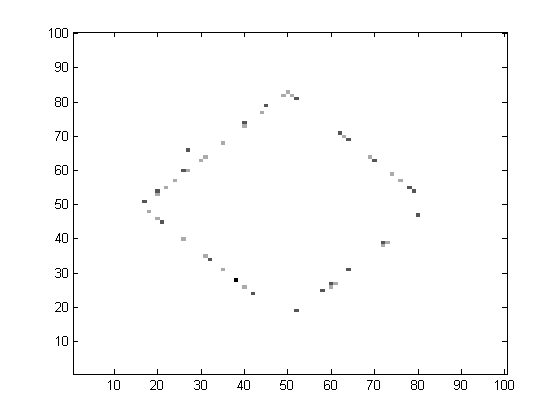}
}%
\subfloat[][h+v, $(4,2)$,\\$\gamma=0$, rel.]{%
\includegraphics[width=0.25\textwidth]{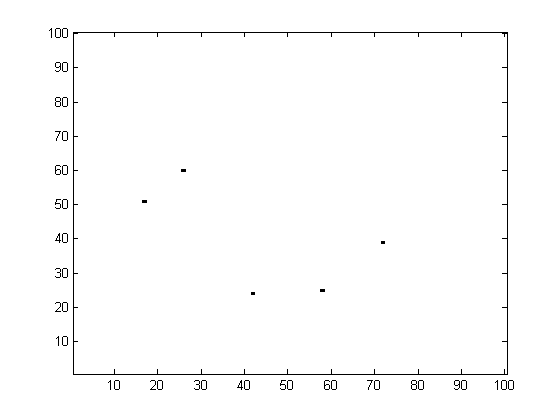}
}%
\subfloat[][h+v, $(4,1)$,\\$\gamma=0$, est.]{%
\includegraphics[width=0.25\textwidth]{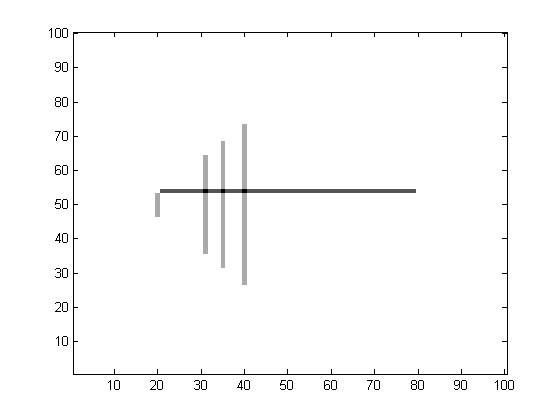}
}%
\subfloat[][h+v, $(4,2)$,\\$\gamma=0$, est.]{%
\includegraphics[width=0.25\textwidth]{blank.png}
}%
\\
\vspace{-10pt}
\subfloat[][h+v, $(4,1)$,\\$\gamma=0.25$, rel.]{%
\includegraphics[width=0.25\textwidth]{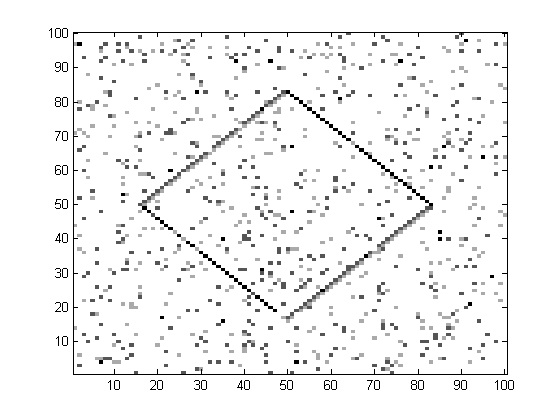}
}%
\subfloat[][h+v, $(4,2)$,\\$\gamma=0$, rel.]{%
\includegraphics[width=0.25\textwidth]{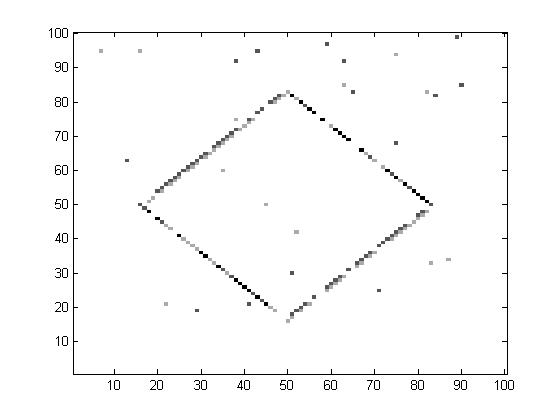}
}%
\subfloat[][h+v, $(4,1)$,\\$\gamma=0.25$, est.]{%
\includegraphics[width=0.25\textwidth]{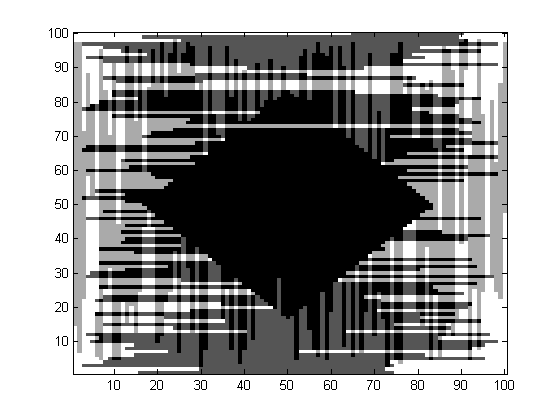}
}%
\subfloat[][h+v, $(4,2)$,\\$\gamma=0.25$, est.]{%
\includegraphics[width=0.25\textwidth]{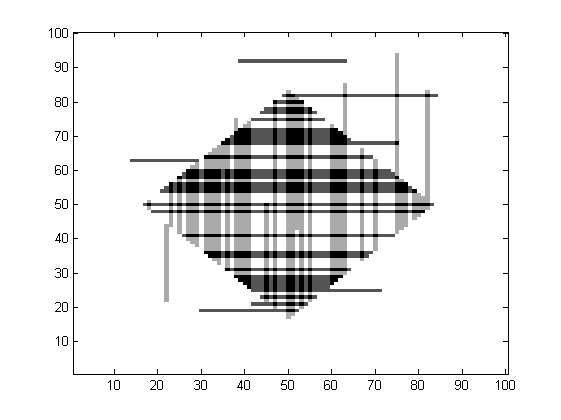}
}%
\\
\vspace{-10pt}
\subfloat[][h+v, $(4,1)$,\\$\gamma=0.49$, rel.]{%
\includegraphics[width=0.25\textwidth]{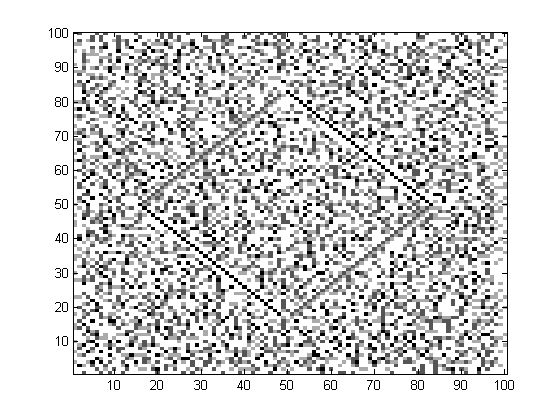}
}%
\subfloat[][h+v, $(4,2)$,\\$\gamma=0.49$, rel.]{%
\includegraphics[width=0.25\textwidth]{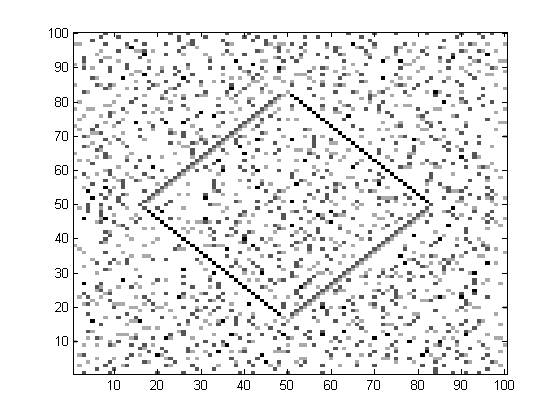}
}%
\subfloat[][h+v, $(4,1)$,\\$\gamma=0.49$, est.]{%
\includegraphics[width=0.25\textwidth]{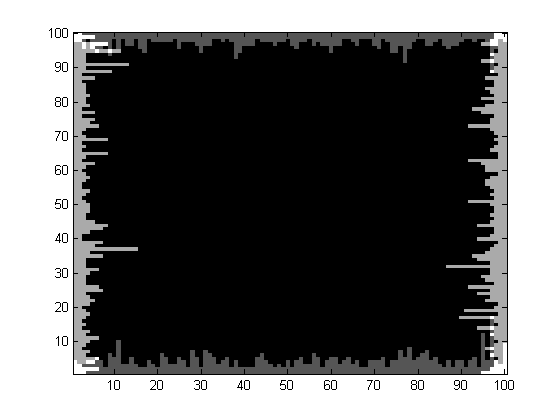}
}%
\subfloat[][h+v, $(4,2)$,\\$\gamma=0.49$, est.]{%
\includegraphics[width=0.25\textwidth]{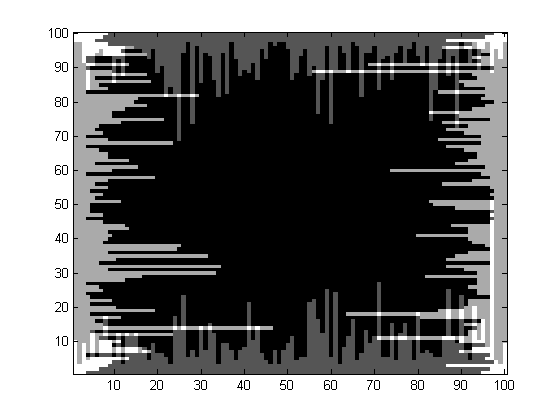}
}%
\caption{The figures show the relevant critical points (rel.) and the corresponding estimates (est.) selected by the horizontal (h) or the horizontal+vertical (h+v) approach using the overlapping ~$(N,Q)$ ~rule and based on a diamond-shaped change set. It is ~$m_k(S)=k+(-1)^k$, $k=1,\ldots,d$ ~and ~$d=500$ ~with ~$\sigma^2=2$}%
\label{fig:sensitivity1}
\end{figure}

\newpage
\begin{figure}[H]%
\captionsetup[subfigure]{labelformat=empty}
\centering  
\vspace{-80pt}
\subfloat[][h, $(4,1)$,\\$\gamma=0$, rel.]{%
\includegraphics[width=0.25\textwidth]{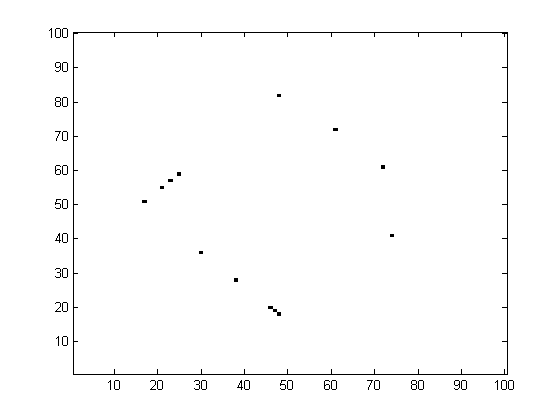}
}%
\subfloat[][h, $(4,2)$,\\$\gamma=0$, rel.]{%
\includegraphics[width=0.25\textwidth]{blank.png}
}%
\subfloat[][h, $(4,1)$,\\$\gamma=0$, est.]{%
\includegraphics[width=0.25\textwidth]{blank.png}
}%
\subfloat[][h, $(4,2)$,\\$\gamma=0$, est.]{%
\includegraphics[width=0.25\textwidth]{blank.png}
}%
\\\vspace{-10pt}
\subfloat[][h, $(4,1)$,\\$\gamma=0.25$, rel.]{%
\includegraphics[width=0.25\textwidth]{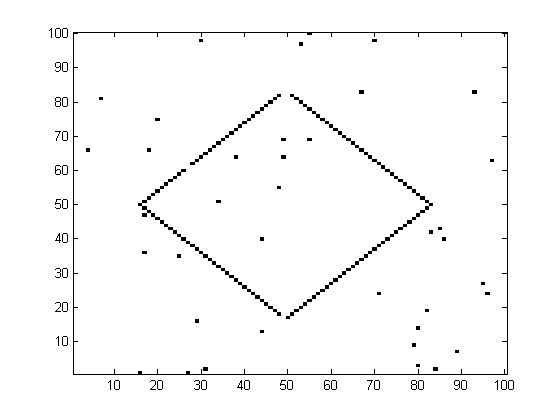}
}%
\subfloat[][h, $(4,2)$,\\$\gamma=0.25$, rel.]{%
\includegraphics[width=0.25\textwidth]{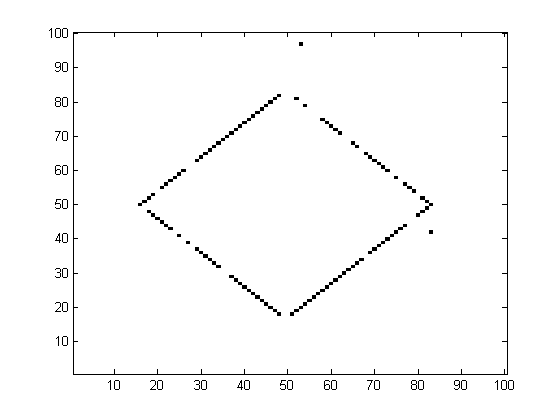}
}%
\subfloat[][h, $(4,1)$,\\$\gamma=0.25$, est.]{%
\includegraphics[width=0.25\textwidth]{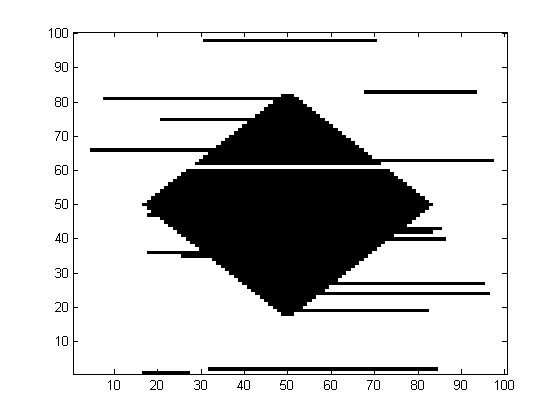}
}%
\subfloat[][h, $(4,2)$,\\$\gamma=0.25$, est.]{%
\includegraphics[width=0.25\textwidth]{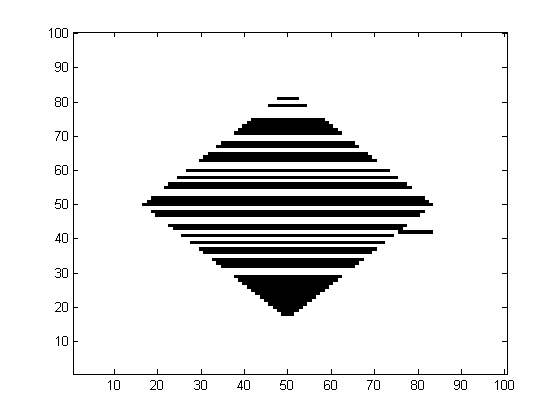}
}%
\\\vspace{-10pt}
\subfloat[][h, $(4,1)$,\\$\gamma=0.49$, rel.]{%
\includegraphics[width=0.25\textwidth]{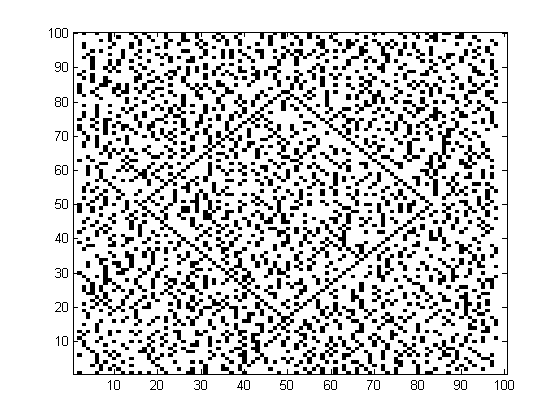}
}%
\subfloat[][h, $(4,2)$,\\$\gamma=0.49$, rel.]{%
\includegraphics[width=0.25\textwidth]{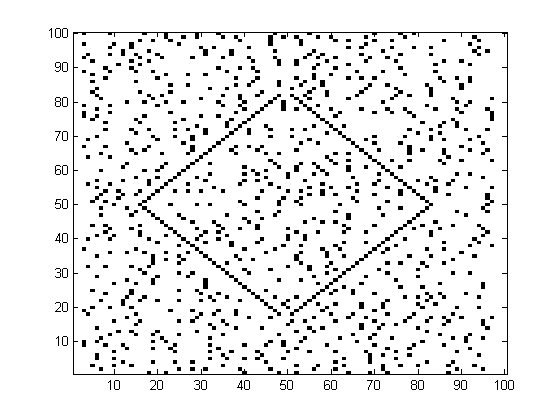}
}%
\subfloat[][h, $(4,1)$,\\$\gamma=0.49$, est.]{%
\includegraphics[width=0.25\textwidth]{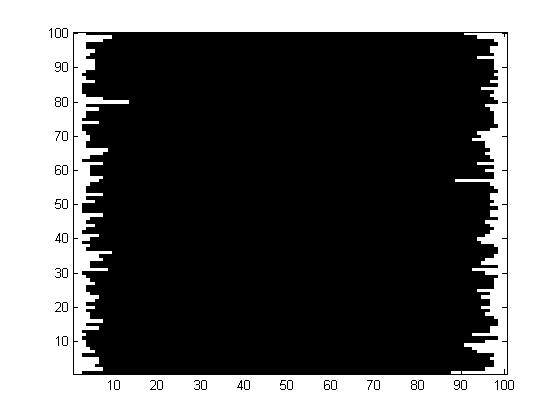}
}%
\subfloat[][h, $(4,2)$,\\$\gamma=0.49$, est.]{%
\includegraphics[width=0.25\textwidth]{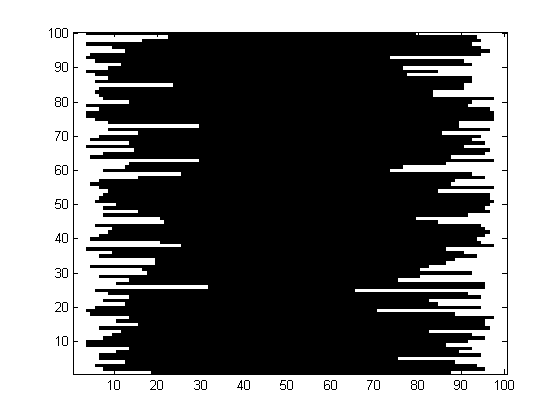}
}%
\\\vspace{10pt}
\subfloat[][h+v, $(4,1)$,\\$\gamma=0$, rel.]{%
\includegraphics[width=0.25\textwidth]{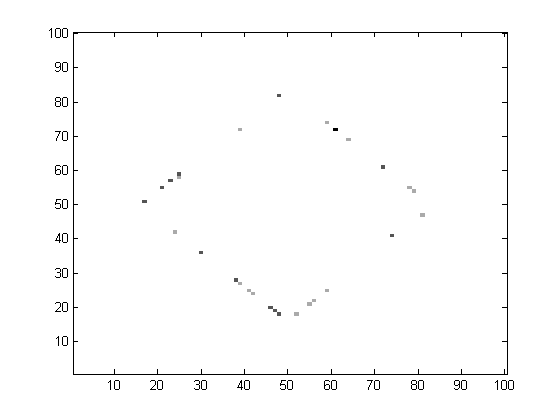}
}%
\subfloat[][h+v, $(4,2)$,\\$\gamma=0$, rel.]{%
\includegraphics[width=0.25\textwidth]{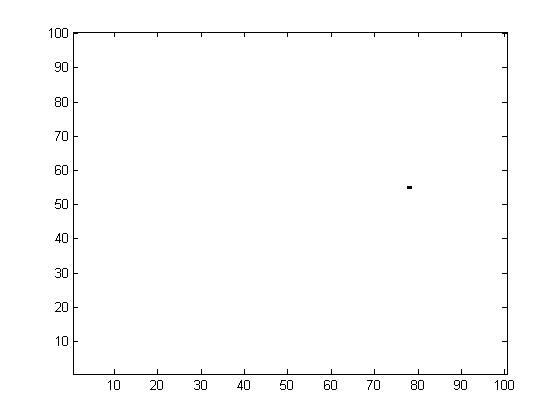}
}%
\subfloat[][h+v, $(4,1)$,\\$\gamma=0$, est.]{%
\includegraphics[width=0.25\textwidth]{blank.png}
}%
\subfloat[][h+v, $(4,2)$,\\$\gamma=0$, est.]{%
\includegraphics[width=0.25\textwidth]{blank.png}
}%
\\\vspace{-10pt}
\subfloat[][h+v, $(4,1)$,\\$\gamma=0.25$, rel.]{%
\includegraphics[width=0.25\textwidth]{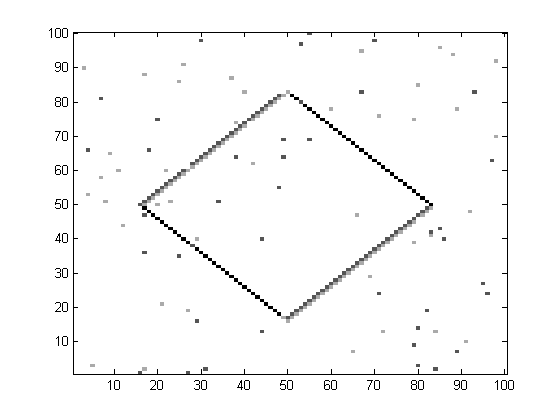}
}%
\subfloat[][h+v, $(4,2)$,\\$\gamma=0.25$, rel.]{%
\includegraphics[width=0.25\textwidth]{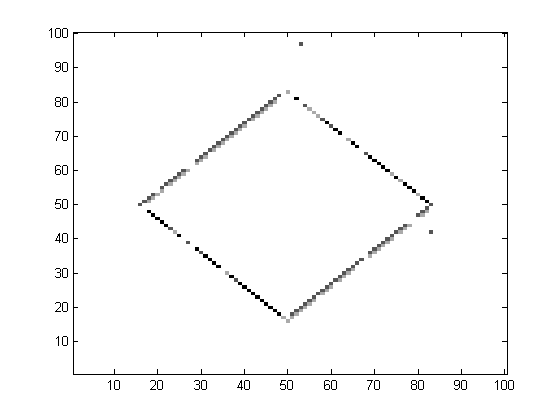}
}%
\subfloat[][h+v, $(4,1)$,\\$\gamma=0.25$, est.]{%
\includegraphics[width=0.25\textwidth]{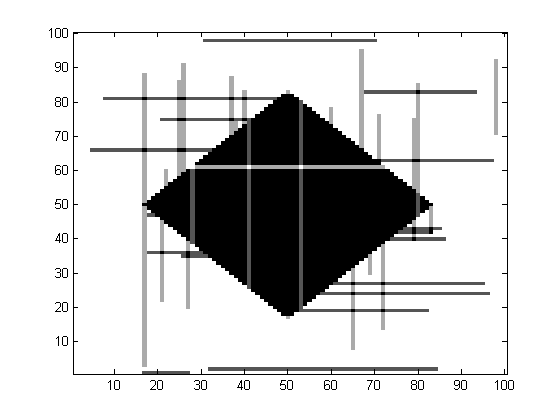}
}%
\subfloat[][h+v, $(4,2)$,\\$\gamma=0.25$, est.]{%
\includegraphics[width=0.25\textwidth]{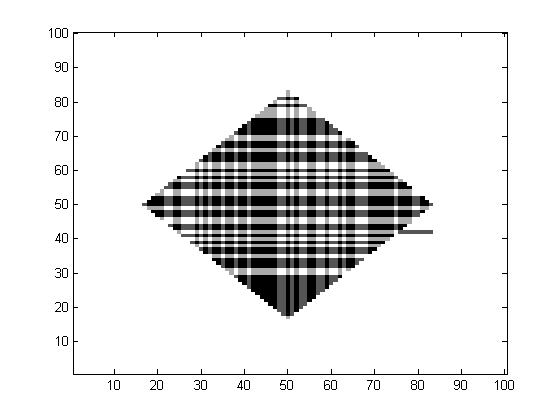}
}%
\\\vspace{-10pt}
\subfloat[][h+v, $(4,1)$,\\$\gamma=0.49$, rel.]{%
\includegraphics[width=0.25\textwidth]{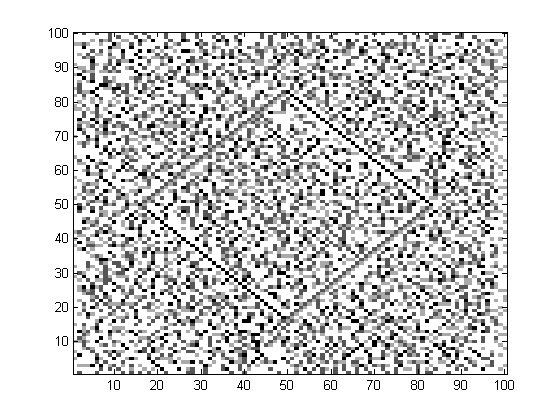}
}%
\subfloat[][h+v, $(4,2)$,\\$\gamma=0.49$, rel.]{%
\includegraphics[width=0.25\textwidth]{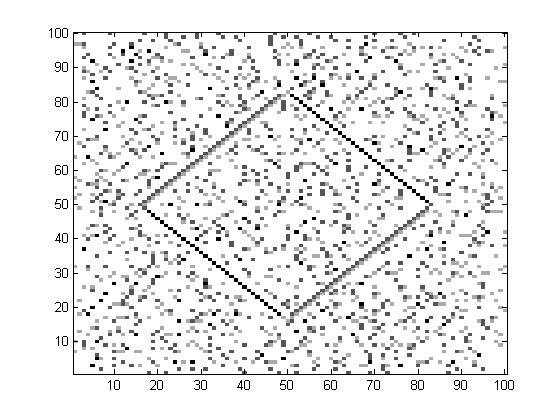}
}%
\subfloat[][h+v, $(4,1)$,\\$\gamma=0.49$, est.]{%
\includegraphics[width=0.25\textwidth]{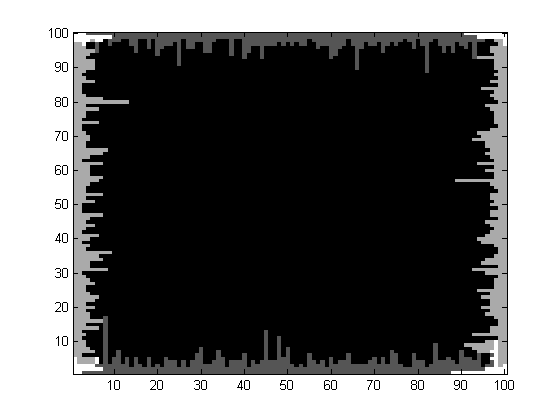}
}%
\subfloat[][h+v, $(4,2)$,\\$\gamma=0.49$, est.]{%
\includegraphics[width=0.25\textwidth]{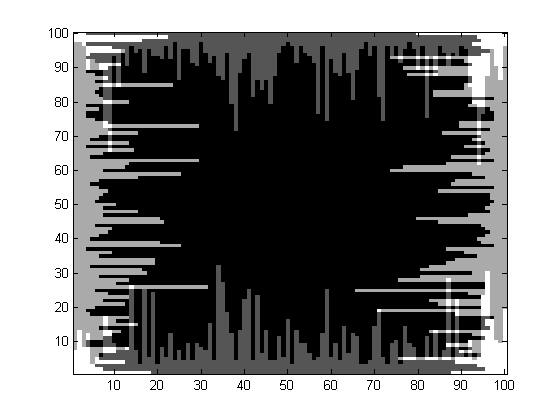}
}%
\caption{The figures show the relevant critical points (rel.) and the corresponding estimates (est.) selected by the horizontal (h) or the horizontal+vertical (h+v) approach using the overlapping ~$(N,Q)$ ~rule and based on a diamond-shaped change set. It is ~$m_k(S)=k+(-1)^k$, $k=1,\ldots,d$ ~and ~$d=1000$ ~with ~$\sigma^2=2$}%
\label{fig:sensitivity2}%
\end{figure}

\newpage
\begin{figure}[H]%
\captionsetup[subfigure]{labelformat=empty}
\centering  
\subfloat[][$(6,1)$, $\gamma=0$]{%
\includegraphics[width=0.3\textwidth]{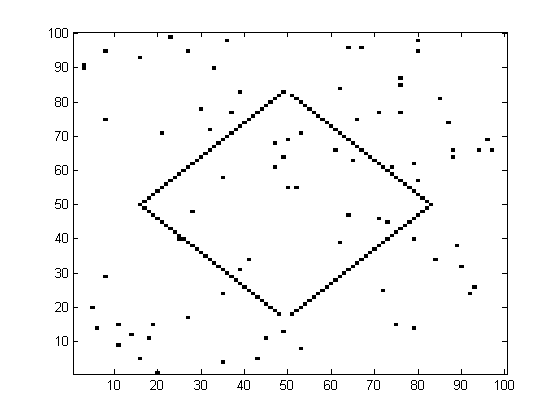}
}%
\subfloat[][$(6,1)$, $\gamma=0.25$]{%
\includegraphics[width=0.3\textwidth]{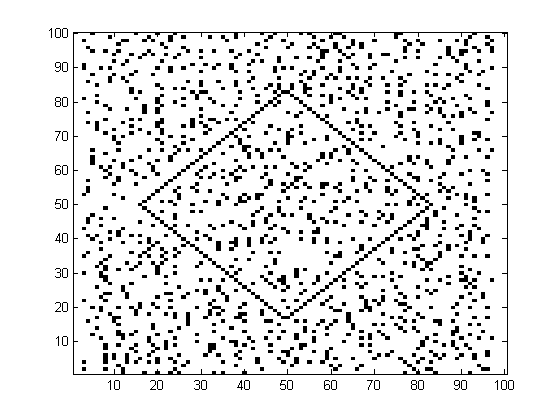}
}%
\subfloat[][$(6,1)$, $\gamma=0.49$]{%
\includegraphics[width=0.3\textwidth]{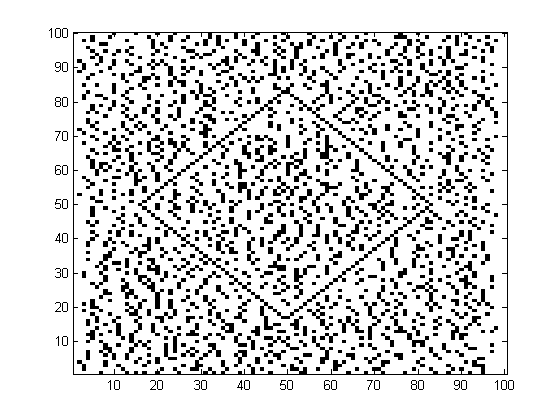}
}%
\\\vspace{-10pt}
\subfloat[][$(6,2)$, $\gamma=0$]{%
\includegraphics[width=0.3\textwidth]{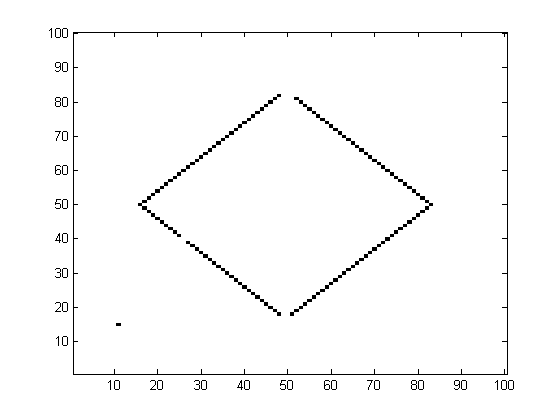}
}%
\subfloat[][$(6,2)$, $\gamma=0.25$]{%
\includegraphics[width=0.3\textwidth]{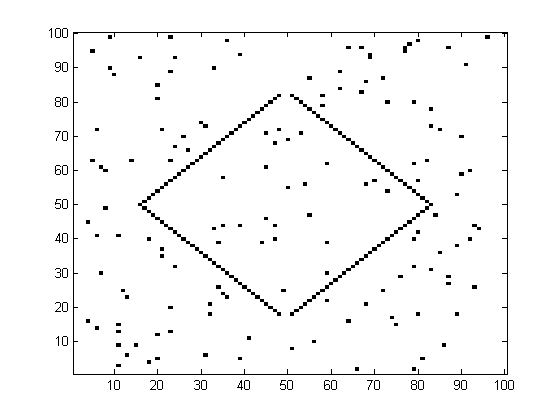}
}%
\subfloat[][$(6,2)$, $\gamma=0.49$]{%
\includegraphics[width=0.3\textwidth]{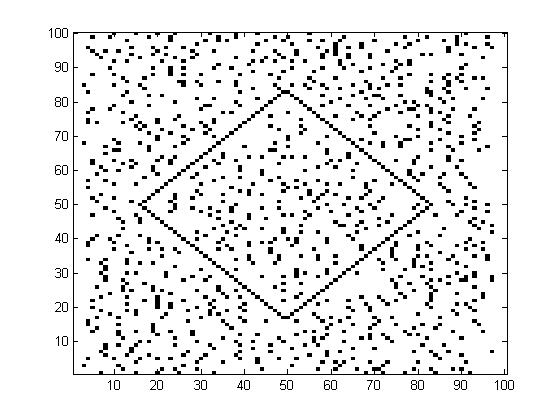}
}%
\\\vspace{-10pt}
\subfloat[][$(6,3)$, $\gamma=0$]{%
\includegraphics[width=0.3\textwidth]{blank.png}
}%
\subfloat[][$(6,3)$, $\gamma=0.25$]{%
\includegraphics[width=0.3\textwidth]{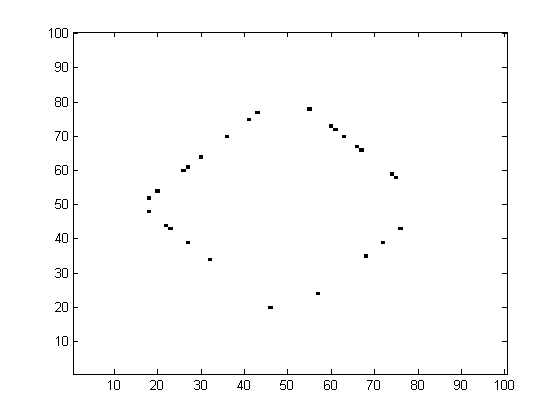}
}%
\subfloat[][$(6,3)$, $\gamma=0.49$]{%
\includegraphics[width=0.3\textwidth]{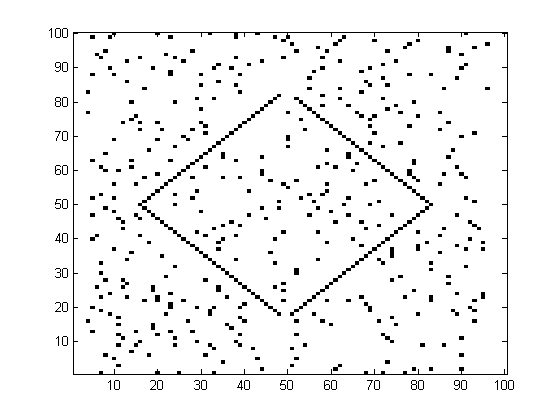}
}%
\\\vspace{-10pt}
\subfloat[][$(6,4)$, $\gamma=0$]{%
\includegraphics[width=0.3\textwidth]{blank.png}
}%
\subfloat[][$(6,4)$, $\gamma=0.25$]{%
\includegraphics[width=0.3\textwidth]{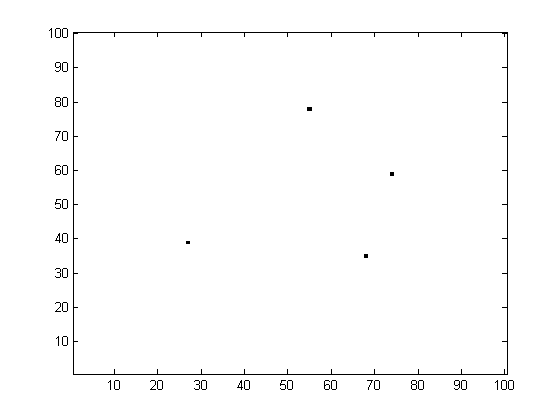}
}%
\subfloat[][$(6,4)$, $\gamma=0.49$]{%
\includegraphics[width=0.3\textwidth]{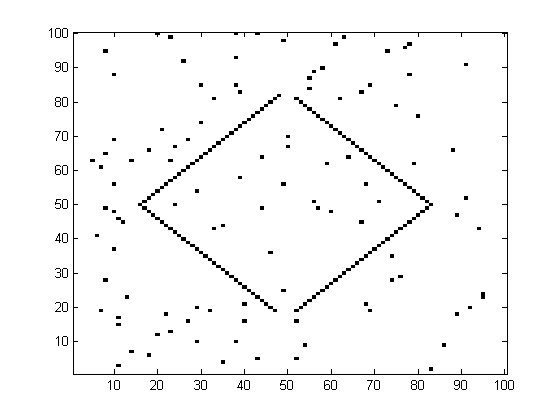}
}%
\caption{The figures show the relevant critical points selected by the horizontal approach using the overlapping ~$(N,Q)$ ~rule and based on a diamond-shaped change set. It is ~$m_k(S)=k+(-1)^k$, $k=1,\ldots,d$ ~and ~$d=1000$ ~with ~$\sigma^2=2$}%
\label{fig:sensitivity3}
\end{figure}

\end{document}